\documentclass[11pt]{amsart}

\usepackage{amsmath}

\usepackage{amsfonts}

\usepackage{amssymb}

\numberwithin{equation}{section}

\usepackage{amsthm}

\newtheorem{thm}{Theorem}[section]

\newtheorem{lemma}[thm]{Lemma}

\newtheorem{remark}[thm]{Remark}

\newtheorem{cor}[thm]{Corollary}

\newtheorem{prop}[thm]{Proposition}

\newtheorem{exam}[thm]{Example}

\newtheorem{defn}[thm]{Definition}

\newtheorem{question}[thm]{Question}

\newcommand\id{\mathop{\rm id}}

\newcommand\sspp{\mathop{\rm span}}

\newcommand\nph{\varphi}

\newcommand\linj{\mathop{\rm el}}

\newcommand\rinj{\mathop{\rm er}}

\newcommand\inj{\mathop{\rm e}}

\newcommand\comm{\mathop{\rm c}}

\newcommand{\cl}[1]{\mathcal{#1}}

\newcommand{\bb}[1]{\mathbb{#1}}

\begin{document}

\title{Quotients, Exactness and Nuclearity in the Operator System Category}

\author[A.~S.~Kavruk]{Ali S.~Kavruk}
\address{Department of Mathematics, University of Houston,
Houston, Texas 77204-3476, U.S.A.}
\email{kavruk@math.uh.edu}

\author[V.~I.~Paulsen]{Vern I.~Paulsen}
\address{Department of Mathematics, University of Houston,
Houston, Texas 77204-3476, U.S.A.}
\email{vern@math.uh.edu}

\author[I.~G.~Todorov]{Ivan G.~Todorov}
\address{Department of Pure Mathematics, Queen's University Belfast,
Belfast BT7 1NN, United Kingdom}
\email{i.todorov@qub.ac.uk}

\author[M.~Tomforde]{Mark Tomforde}
\address{Department of Mathematics, University of Houston, Houston,
  Texas 77204-3476, U.S.A.}
\email{tomforde@math.uh.edu}


\date{\today}

\thanks{The first and second authors were supported in part by NSF
  grant DMS-0600191. The fourth author was supported by NSA Grant H98230-09-1-0036.}

\begin{abstract}
We continue our study of tensor products in the operator system
category. We define operator system quotients and exactness in this
setting and refine the notion of nuclearity by studying operator
systems that preserve various pairs of tensor products. One of our
main goals is to relate these refinements of nuclearity to the Kirchberg conjecture. In
particular, we prove that the Kirchberg conjecture is equivalent to
the statement that every operator system that is (min,er)-nuclear is
also (el,c)-nuclear. We show that operator system quotients are not
always equal to the corresponding operator space quotients and then study
exactness of various operator system tensor products for the operator
system quotient. We prove that an operator system is exact for the min
tensor product if and only if it is (min,el)-nuclear. We give many
characterizations of operator systems that are (min,er)-nulcear,
(el,c)-nuclear, (min,el)-nulcear and (el,max)-nuclear. These
characterizations involve operator system analogues of various
properties from the theory of C*-algebras and operator spaces,
including the WEP and LLP.
\end{abstract}

\maketitle

\section{Introduction}

In this paper we continue our study of tensor products in the operator
system category that we started in \cite{kptt}. We exhibit analogues of many results that have
been obtained earlier for C*-algebras and operator spaces. One of our main objectives is
to relate operator system tensor products with the Kirchberg conjecture \cite{ki}; we achieve this
in Section \ref{s_llp} where we obtain equivalences of the conjecture with statements
about operator system tensor products.

We begin by constructing a quotient that is
appropriate for the category of operator systems and completely
positive maps. Thus, given an operator system $\cl S$ and a
kernel $J\subseteq \cl S$ of a completely positive map $\phi$ defined on $\cl S$, we
construct a quotient operator system $\cl S/J$ with the property that
$\phi$ factors through that quotient to a completely positive map.
Every operator system is also an
operator space and there is a definition of quotients of operator
spaces. Thus, $\cl S/J$ has two natural operator space structures:
firstly, $\cl S/J$ is the operator space that arises
as the quotient of two operator spaces, and secondly, it is also the operator
space, induced by our operator system quotient. We show that these two
operator space structures on $\cl S/J$ are, generally, not
boundedly isomorphic. We then develop the concept of an exact
operator system in this context and give a characterization of
exactness in terms of equality of certain tensor products.

Next, we study the weak expectation property (WEP) in the
operator system setting.
Recall that the weak expectation property for C*-algebras has a
characterization in terms of the equality of certain tensor products.
Namely, a C*-algebra $\cl A$ has the WEP if and only if
$\cl A \otimes_{\max} \cl B \subseteq I(\cl A) \otimes_{\max} \cl B$, for every C*-algebra $\cl B$
(here $I(\cl A)$ is the injective envelope of $\cl A$).
In our earlier work \cite{kptt}, we showed that the maximal tensor product on
the category of C*-algebras has two distinct natural extensions to the
category of operator systems.  Thus, we are lead to study multiple
``natural'' extensions of WEP to operator systems,
together with their relationships and characterizations in terms of the equality of various tensor
products.

In a similar vein, we show that some operator system analogues of the
local lifting property (LLP) are equivalent to equality of certain
operator system tensor products. Via these considerations we obtain
some equivalences of the Kirchberg Conjecture in terms of operator
system tensors. Among these, we show that Kirchberg's Conjecture is
equivalent to the fact that every operator system that satisfies an
operator system analogue of the LLP, which we call the OSLLP, satisfies an
operator system analogue of the WEP, that we call the {\em double commutant
expectation property.}


\section{Preliminaries}\label{s_prel}


In this section we establish the terminology and state the definitions that shall be used
throughout the paper.

A {\bf $*$-vector space} is a complex vector space $V$ together with
a map $^* : V \to V$ that is involutive (i.e., $(v^*)^* = v$ for
all $v \in V$) and conjugate linear (i.e., $(\lambda v + w)^* =
\overline{\lambda} v^* + w^*$ for all $\lambda \in \bb C$ and all $v,w
\in V$). If $V$ is a $*$-vector space, then we let $V_h = \{x \in V
: x^* = x \}$ and we call the elements of $V_h$ the {\bf hermitian}
elements of $V$. Note that $V_h$ is a real vector space.

An {\bf ordered $*$-vector space} is a pair $(V, V^+)$ consisting of a $*$-vector space $V$ and a subset $V^+ \subseteq V_h$ satisfying
the following two properties:

\medskip

(a) $V^+$ is a spanning cone in $V_h$;

(b) $V^+ \cap -V^+ = \{ 0 \}$.

\medskip

In any ordered $*$-vector space we may define a partial order $\geq$ on $V_h$
by letting $v \geq w$ (or, equivalently, $w \leq v$) if and only if $v - w \in V^+$.
Note that $v \in V^+$ if and only if $v
\geq 0.$ For this reason $V^+$ is called the cone of
{\bf positive} elements of $V$.

If $(V,V^+)$ is an ordered $*$-vector space, an element
$e \in V_h$ is called an {\bf order unit} for $V$ if
for all $v \in V_h$ there exists a real number $r > 0$
such that $re \geq v$. In fact, if $V^+$ is a cone and $(V,V^+)$ has an order unit, then
$V^+$ is automatically spanning, so the spanning condition is often
not included in (a).

If $(V, V^+)$ is an ordered $*$-vector space with an order unit $e$,
then we say that $e$ is an {\bf Archimedean order unit} if whenever
$v \in V$ and $re+v \geq 0$ for all real $r >0$, we have that $v \in V^+$.
In this case, we call the triple $(V, V^+, e)$ an {\bf Archimedean
ordered $*$-vector space} or an {\bf AOU space,} for short.
The {\bf state space} of $V$ is the set $S(V)$ of all linear maps $f : V\rightarrow\bb{C}$
such that $f(V^+)\subseteq [0,\infty)$ and $f(e) = 1$.

If $V$ is a $*$-vector space, we let $M_{m,n}(V)$ denote the set of
all $m \times n$ matrices with entries in $V$ and set $M_{n}(V) =
M_{n,n}(V)$.  The entry-wise
addition and scalar multiplication turn $M_{m,n}(V)$ into a complex
vector space. We set $M_{m,n} = M_{m,n}(\bb{C})$, $M_n = M_{n,n}$, and let
$\{E_{i,j} : 1\leq i\leq m, 1\leq j \leq n\}$ denote the canonical
matrix unit system of $M_{m,n}$. If $X = (x_{i,j})_{i,j} \in M_{l,m}$ is a scalar
matrix, then for any $A = (a_{i,j})_{i,j} \in M_{m,n}(V)$ we let
$XA$ be the element of $M_{l,n}(V)$ whose $i,j$-entry $(XA)_{i,j}$
equals $\sum_{k=1}^m x_{i,k} a_{k,j}$. We define multiplication by
scalar matrices on the left in a similar way. Furthermore, when
$m=n$, we define an involution on $M_n(V)$ by letting
$(a_{i,j})_{i,j}^* := (a_{j,i}^*)_{i,j}$. With respect to this
operation, $M_n(V)$ is a $*$-vector space. We let $M_n(V)_h$ be the
set of all hermitian elements of $M_n(V)$.

\begin{defn}
Let $V$ be a $*$-vector space.  We say that $\{ C_n \}_{n=1}^\infty$ is a \emph{matrix ordering} on $V$ if
\begin{enumerate}
\item $C_n$ is a cone in $M_n(V)_h$ for each $n\in \bb{N}$,
\item $C_n \cap -C_n = \{0 \}$ for each $n\in \bb{N}$, and
\item for each $n,m\in \bb{N}$ and $X \in M_{n,m}$
we have that $X^* C_n X \subseteq C_m$.
\end{enumerate}
In this case we call $(V, \{C_n \}_{n=1}^\infty )$ a \emph{ matrix
  ordered $*$-vector space.} We refer to condition (3) as the
  \emph{compatibility of the family
  $\{C_n\}_{n=1}^{\infty}$}.
\end{defn}

Note that conditions (1) and (2) show that $(M_n(V), C_n)$
is an ordered $*$-vector space for each $n\in \bb{N}$. As usual, when $A,B \in M_n(V)_h$, we
write $A \leq B$ if $B-A \in C_n$.

\begin{defn}\label{op-system-abstract-def}
Let $(V, \{ C_n \}_{n=1}^\infty)$ be a matrix ordered $*$-vector space.  For $e \in V_h$ let
$$e_n := \left( \begin{smallmatrix} e & & \\ & \ddots & \\ & & e \end{smallmatrix} \right)$$
be the corresponding diagonal matrix in $M_n(V)$.
We say that $e$ is a \emph{matrix order unit} for $V$
if $e_n$ is an order unit for $(M_n(V), C_n)$ for each $n$.
We say that $e$ is an \emph{Archimedean matrix order unit} if
$e_n$ is an Archimedean order unit for $(M_n(V), C_n)$ for each $n$.
An \emph{(abstract) operator system} is a triple $(V, \{C_n \}_{n=1}^\infty, e)$,
where $V$ is a complex $*$-vector space,
$\{C_n \}_{n=1}^\infty$ is a matrix ordering on $V$,
and $e \in  V_h$ is an Archimedean matrix order unit.
\end{defn}

The above definition of an operator system was first
introduced by Choi and Effros in \cite{ce}.
We note that the dual of an operator system is a matrix-ordered space in a canonical fashion \cite{ce}.
If $V $ and $V'$ are
vector spaces and $\phi : V \to V'$ is a linear map, then for each
$n\in \bb{N}$ the map $\phi$ induces a linear map $\phi^{(n)} : M_n(V)
\to M_n(V')$ given by $\phi^{(n)} ((v_{i,j})_{i,j}) :=
(\phi(v_{i,j}))_{i,j}$. If $(V, \{ C_n \}_{n=1}^\infty)$ and $(V',
\{ C_n' \}_{n=1}^\infty)$ are matrix ordered $*$-vector spaces, a
map $\phi : V \to V'$ is called {\bf completely positive}  if
$\phi^{(n)}(C_n) \subseteq C_n'$ for each $n\in \bb{N}$. Similarly, we
call a linear map $\phi : V \to V'$ a {\bf complete order
isomorphism} if $\phi$ is invertible and both $\phi$ and $\phi^{-1}$
are completely positive.
The following easy fact, whose proof we omit, will be used repeatedly in the paper.

\begin{lemma}\label{comp}
Let $\cl R_1,$ $\cl R_2$ and $\cl R_3$ be operator systems and $f : \cl
R_1\rightarrow \cl R_2$ and  $g : \cl R_2\rightarrow \cl R_3$ be unital
completely positive maps. If $g\circ f :\cl R_1\rightarrow \cl R_3$ is
a complete order isomorphism (onto its range) then $f$ is a complete order isomorphism (onto its range).
\end{lemma}

Let $\cl B(H)$ be the space of all bounded linear operators
acting on a Hilbert space $H$.
The direct sum of $n$ copies of $H$ is denoted by $H^n$ and its elements are written as column vectors.
A {\bf concrete operator system}
$\cl S$ is a subspace of $\cl B (H)$ such that $\cl S = \cl S^*$ and
$I \in \cl S$. (Here, and in the sequel, we let $I$ denote the
identity operator.) As is the case for many classes of subspaces
(and subalgebras) of $\cl B (H)$, there is an abstract
characterization of concrete operator systems.
If $\cl S \subseteq \cl B (H)$
is a concrete operator system, then we observe that $\cl S$ is a
$*$-vector space with respect to the adjoint operation, $\cl S$ inherits an order structure from $\cl B(H)$,
and has $I$ as an Archimedean order unit. Moreover, since $\cl
S \subseteq \cl B(H)$, we have that $M_n(\cl S) \subseteq M_n( \cl B(H)) \equiv \cl B (H^n)$ and hence $M_n(\cl S)$ inherits
an involution and an order
structure from $\cl B ( H^n)$ and has the $n \times n$ diagonal matrix
$$\begin{pmatrix} I & & \\ & \ddots & \\ & & I \end{pmatrix}$$ as an
Archimedean order unit.  In other words, $\cl S$ is
an abstract operator system in the sense of Definition
\ref{op-system-abstract-def}. The following result of Choi and
Effros \cite[Theorem~4.4]{ce} shows that the converse is also true.
For an alternative proof of the result, we refer the reader to
\cite[Theorem~13.1]{pa}.

\begin{thm}[Choi-Effros]
\label{th-choieffros}
Every concrete operator system $\cl S$ is an abstract operator system.
Conversely, if $(V, \{C_n \}_{n=1}^\infty, e)$ is an abstract operator system, then there
exists a Hilbert space $ H$, a concrete operator system
$\cl S \subseteq \cl B(H)$, and a complete order
isomorphism $\phi : V \to \cl S$ with $\phi(e) = I$.
\end{thm}

Thanks to the above theorem, we can identify abstract and concrete
operator systems and refer to them simply as operator systems.
To avoid excessive notation, we will generally refer to an operator
system as simply a set $\cl S$ with the understanding that $e$ is the
order unit and $M_n(\cl S)^+$ is the cone of positive elements
in $M_n(\cl S)$. We note that any unital C*-algebra (and all C*-algebras in the
present paper will be assumed to be unital) is also an operator system in a canonical way.

There is a similar theory for arbitrary subspaces $X \subseteq \cl
B(H)$, called {\bf concrete operator spaces}. The
identification $M_n(\cl B(H)) \equiv \cl B(H^n)$ endows each $M_n(X)
\subseteq M_n(\cl B(H))$ with a norm; the family of norms obtained
in this way
satisfies certain compatibility axioms called {\it
Ruan's axioms}, and
is called an operator space structure. Ruan's Theorem identifies the vector spaces
satisfying Ruan's axioms with the concrete operator spaces. Sources
for the details include \cite{er2} and \cite{pa}.
If $\cl S$ is an abstract operator system and
$\phi : \cl S\rightarrow \cl B(H)$ is a complete order isomorphism onto its range,
then the concrete operator space structure on $\phi(\cl S)$ can be pulled back and will thus yield
an operator space structure on $\cl S$.
Moreover,
this operator space structure does not depend on the particular inclusion $\phi$.
Thus, the operator system structure on $\cl S$ {\em induces} an operator space structure on $\cl S$.

We close this section with some notions and notation concerning operator system tensor products \cite{kptt}.
Let
$(\cl S, \{ P_n\}_{n=1}^{\infty}, e_1)$ and $ (\cl T,\{
Q_n\}_{n=1}^{\infty}, e_2)$ be operator systems.
We equip the algebraic tensor product $\cl S\otimes\cl T$ with the involution given by
$(x\otimes y)^* = x^* \otimes y^*$.
By an {\bf operator system structure
on $\cl S \otimes \cl T$} we mean a family $\tau = \{C_n\}_{n=1}^{\infty}$
of cones, where $C_n \subseteq M_n(\cl S \otimes \cl T)_{h}$, satisfying:

\begin{itemize}
\item[(T1)] $(\cl S \otimes \cl T,\{C_n\}_{n=1}^{\infty}, e_1 \otimes e_2)$ is
an operator system denoted $\cl S \otimes_{\tau} \cl T$,
\item[(T2)] $P_n \otimes Q_m \subseteq C_{nm},$ for all $n,m \in \bb
N$, and
\item[(T3)] if $\phi:\cl S \to M_n$  and $\psi:\cl T \to M_m$ are unital completely
positive maps, then $\phi \otimes \psi:\cl S \otimes_{\tau} \cl T \to M_{mn}$ is a unital completely
positive map.
\end{itemize}
To simplify notation we shall generally write $C_n = M_n(\cl S \otimes_{\tau} \cl T)^+.$

If we let $\cl O$ denote the category whose objects are operator
systems and whose morphisms are the unital completely positive maps,
then by an {\bf operator system tensor product,} we mean a
mapping  $\tau: \cl O \times \cl O \to \cl O,$
such that for every pair of operator systems $\cl S$ and $\cl T,$
 $\tau(\cl S, \cl T)$ is an operator system structure on $\cl
S \otimes \cl T,$ denoted $\cl S\otimes_{\tau}\cl T$.
We call an operator system tensor product $\tau$ {\bf functorial,}
if the following property is satisfied:

\begin{itemize}
\item[(T4)] for any four operator systems $\cl S_1,\cl S_2, \cl T_1$, and $\cl T_2$, we have that if $\phi : \cl S_1\rightarrow \cl S_2$ and $\psi : \cl T_1\rightarrow \cl
T_2$ are unital completely positive maps, then the linear map $\phi\otimes\psi : \cl S_1\otimes\cl
T_1\rightarrow\cl S_2\otimes\cl T_2$ is (unital and) completely positive.
\end{itemize}

We will use extensively throughout the exposition several of the tensor products studied in \cite{kptt}.
Let $\cl S$ and $\cl T$ be operator systems. Their {\bf minimal} tensor product $\cl S\otimes_{\min}\cl T$ arises
from the embedding $\cl S\otimes \cl T\subseteq \cl B(H\otimes K)$, where $\cl S \subseteq \cl B(H)$ and $\cl T\subseteq \cl B(K)$
are any embeddings given by the Choi-Effros Theorem.
The {\bf maximal} tensor product $\cl S\otimes_{\max}\cl T$
is characterized by the property that every jointly completely positive bilinear map
$\rho$ from $\cl S\times \cl T$ into an operator system $\cl R$ linearizes to a completely positive
linear map from $\cl S\otimes_{\max}\cl T$ into $\cl R$ (see \cite[Definitions 5.4 and 5.6]{kptt}). Finally, to describe
the {\bf commuting} tensor product $\cl S\otimes_{\comm}\cl T$, note that
any pair $\phi : \cl S\rightarrow \cl B(H)$ and $\psi : \cl T\rightarrow \cl B(H)$ of completely positive maps
with commuting ranges gives rise to a completely positive map $\phi\cdot\psi : \cl S\otimes\cl T\rightarrow \cl B(H)$ given by
$(\phi\cdot\psi)(x\otimes y) = \phi(x)\psi(y)$.
The positive cone $M_n(\cl S\otimes_{\comm}\cl T)^+$
consists by definition of those elements $u\in M_n(\cl S\otimes\cl T)$ for which
$(\phi\cdot\psi)^{(n)}(u) \geq 0$ for all pairs of completely positive maps $\phi : \cl S\rightarrow \cl B(H)$
and $\psi : \cl T\rightarrow \cl B(H)$ with commuting ranges, and all Hilbert spaces $H$.

If $\cl S$ is an operator system, there exists a (unique)
C*-algebra $C_{u}^*(\cl S)$ generated by $\cl S$, called the universal (or maximal)
C*-algebra of $\cl S$, satisfying the following property:
whenever $\cl A$ is a C*-algebra and $\phi: \cl S \to \cl A$ is a unital
completely positive map, there exists a *-homomorphism $\rho: C_u^*(\cl S) \to \cl A$ extending $\phi$.
To construct this C*-algebra, one starts with the free $*$-algebra
$$\cl F(\cl S) = \cl S \oplus (\cl S \otimes \cl S) \oplus
(\cl S \otimes \cl S \otimes \cl S) \oplus \cdots.$$ Each unital
completely positive map $\phi: \cl S \to \cl B(H)$ gives rise to a
$*$-homomorphism $\pi_{\phi}: \cl F(\cl S) \to \cl B(H)$ by setting
\[\pi_{\phi}(s_1 \otimes \cdots \otimes s_n) = \phi(s_1) \cdots \phi(s_n)\]
and extending linearly to the tensor product and then to the direct sum.
For $u\in \cl F(\cl S)$, one sets $\|u\|_{\cl F(\cl S)} = \sup\|\pi_{\phi}(u)\|$,
where the supremum is taken over all unital completely positive maps $\phi$ as above,
and defines $C^*_u(\cl S)$ to be the completion of $\cl F(\cl S)$ with respect to $\|\cdot\|_{\cl F(\cl S)}$.
We will identify $\cl S$ with its image in $C^*_u(\cl S)$, and thus consider
it as an operator subsystem of $C^*_u(\cl S)$.


The following lemma will be used repeatedly later; part (1) is
\cite[Theorem 6.7]{kptt}, part (2) is \cite[Theorem 6.4]{kptt},
while part (3) can be proved similarly to (2).

\begin{lemma}\label{l_cintom}
Let $\cl S$ and $\cl T$ be operator systems.
\begin{enumerate}
\item $\cl S\otimes_{\comm}\cl A = \cl S\otimes_{\max}\cl A$ for every C*-algebra $\cl A$;

\item $\cl S\otimes_{\comm} \cl T$ coincides with the operator system arising from the inclusion of $\cl S\otimes\cl T$ into
$C^*_u(\cl S)\otimes_{\max} C^*_u(\cl T)$;

\item $\cl S\otimes_{\comm} \cl T$ coincides with the operator system arising from the inclusion of $\cl S\otimes\cl T$ into
$\cl S\otimes_{\comm} C^*_u(\cl T)$.
\end{enumerate}
\end{lemma}

If $\tau$ is an operator system tensor product and $\cl S$ and $\cl T$ are operator systems,
we denote by $\cl S\hat{\otimes}_{\tau}\cl T$ the completion of $\cl S\otimes_{\tau}\cl T$
in the norm $\|\cdot\|_{\tau}$ induced by its operator system structure.
If $\cl U\subseteq \cl S\otimes_{\tau}\cl T$, we will denote by $\overline{\cl U}^{\tau}$ the closure
of $\cl U$ in the topology of the norm $\|\cdot\|_{\tau}$.
If $\tau$ is moreover a functorial operator system tensor product, $\cl S,\cl S_1$, $\cl T,\cl T_1$ are operator systems and
$\phi : \cl S\rightarrow \cl S_1$ and $\psi : \cl T\rightarrow \cl T_1$ are unital completely positive mappings,
the mapping $\phi \otimes\psi : \cl S\otimes_{\tau}\cl S_1\rightarrow \cl T\otimes_{\tau}\cl T_1$ is completely contractive
and hence extends to a (completely contractive)
mapping from $\cl S\hat{\otimes}_{\tau}\cl S_1$ into $\cl T\hat{\otimes}_{\tau}\cl T_1$, which will be denoted in
the same way.


\section{Quotients of Operator Systems}\label{s_qos}

In this section, we introduce operator system quotients and establish
some of their properties.
Let $\cl S$ and $\cl T$ be operator systems and $\phi: \cl S \to \cl T$ be a non-zero completely positive map.
Then the kernel of $\phi$, denoted $\ker\phi,$ is a closed,  non-unital $*$-subspace of
$\cl S$, that is, a closed subspace of $\cl S$
that does not contain the unit $e$ of $\cl S$ and has the property that if
an element $x$ of $\cl S$ belongs to it, then so does $x^*.$
However, these properties do not characterize kernels of completely positive maps.
Indeed, any completely positive map $\phi : M_n\rightarrow \cl T$ with
$\phi(E_{1,1}) = 0$ must, by the Cauchy-Schwarz inequality, satisfy $\phi(E_{1,j}) = \phi(E_{i,1}) =0$ for every $i,j = 1,\dots,n$.
Thus, we have that the span $J \subseteq M_n$ of the matrix unit $E_{1,1}$ is closed, non-unital and selfadjoint, but
it is not a kernel of a completely positive map on $M_n$.
Before we introduce the notion of a quotient in the operator system category, it will thus be
convenient to have a characterization of kernels of completely positive maps.
The following result is immediate:

\begin{prop}\label{eq-lem}
Let $\cl S$ be an operator system and let $J \subseteq \cl S$ be a subspace. Then the following are equivalent:
\begin{enumerate}

\item there exists an operator system $\cl T$ and a unital completely positive map $\phi : \cl S\rightarrow \cl T$
such that $J = \ker\phi$;

\item
there exist operator systems $\cl T_{\alpha}$, and
unital completely positive maps $\phi_{\alpha} : \cl S\rightarrow \cl T_{\alpha}$, $\alpha\in \bb{A}$,
such that $J = \bigcap_{\alpha\in \bb{A}}\ker\phi_{\alpha}$;

\item
there exists an operator system $\cl T$ and a non-zero completely positive map $\phi : \cl S\rightarrow \cl T$
such that $J = \ker\phi$;

\item there exists an operator system $\cl T$ and a non-zero positive map $\phi : \cl S\rightarrow \cl T$
such that $J = \ker\phi$;

\item there exists a collection $\{f_{\alpha}\}_{\alpha\in \bb{A}}$ of states of $\cl S$ such that $J
= \bigcap_{\alpha\in \bb{A}} \ker f_{\alpha}$.\\
\end{enumerate}
In all of these cases, $J$ is closed, non-unital, $*$-invariant, and is equal to the intersection of the kernels of all the states that vanish on $J$.
\end{prop}

\begin{defn}\label{d_kern} Given an operator system $\cl S,$ we call $J \subseteq \cl S$ a {\bf kernel,} provided that it satisfies any of the equivalent conditions of Proposition \ref{eq-lem}.
\end{defn}

\noindent {\bf Remarks (i)}
A subspace $J \subseteq \cl S$ is called an {\em order ideal} if $q \in J$ and $0 \le p \le q,$ imply that $p \in J.$
It is easy to see that every kernel is an order ideal.
The example given before Proposition \ref{eq-lem} shows that an order ideal need not be a kernel.

\bigskip

\noindent {\bf (ii)}
Let
$\cl S$ be an operator system, $J\subseteq \cl S$ be any subspace which is invariant under the involution
and does not contain the unit $e$ of $\cl S$, and $q : \cl S\rightarrow \cl S/J$ be the natural map onto the algebraic quotient.
The space $\cl S/J$ is a $*$-vector space with respect to the involution $(x+J)^* = x^*+J$.
Let $D_n = D_n(\cl S/J) = q^{(n)}(M_n(\cl S)^+)$, $n\in \bb{N}$; in other words,
$$
D_n =\{ (x_{ij}+J)\in M_n(\cl S/J): \mbox{ there exists } y
_{ij} \in  J \mbox{ with } (x_{ij}+y_{ij})\in M_n(\cl S)^+\}.
$$
For every $n\in \bb{N}$, the set $D_n$ is a cone in $M_n(\cl S/J)_{h}$ and the collection
$\{D_n\}_{n\in \bb{N}}$ is compatible. However, it is often the case
that $D_n \cap (-D_n)$ contains non-zero elements. It is easy to show that $e+J$ is a
matrix order unit for $\cl S/J$.
It follows that the quotient map $q: \cl S \to \cl S/J$ is a unital completely positive map with kernel equal to $J.$
Thus, we see that any non-zero non-unital selfadjoint subspace of any operator system is the kernel of a completely positive
map into a space resembling a matrix ordered space with a matrix order unit.

When $J \subseteq \cl S$ is an order ideal, then one can show that
$D_n \cap (-D_n) = (0)$ and that $\cl S/J$ is in fact a matrix-ordered
space with matrix order unit $e+J.$
It follows that the conditions in Proposition \ref{eq-lem} can not be relaxed
by requiring that the ranges of the maps be merely matrix ordered
spaces with a matrix order unit instead of operator systems.

We thus have that when $J$ is an order ideal and
$e+J$ is an Archimedean matrix order unit for $\cl S/J$, then
the matrix ordered space $\cl S/J$ is an operator system.
This shows that in the example before Proposition \ref{eq-lem}, where $J$ is the span of
the matrix unit $E_{1,1}$ in $M_n$, then $M_n/J$ is a matrix-ordered
space and the element $I_n +J$ is a matrix order unit of $M_n/J,$ but
it is not Archimedean.

\bigskip

Let $(Q,\{D_n\}_{n=1}^{\infty},1)$ be any matrix ordered space with matrix order unit
$1$, which is not assumed to be Archimedean. It was shown in \cite{pt} and \cite{ptt} that
$Q$ gives rise to an operator system through an
\textit{Archimedeanization process}, which we wish to apply to construct an operator system $\cl S/J,$ when $J$ is a kernel in an
operator system $\cl S$.
The Archimedeanization process involves two steps. In the first step, one identifies a subspace $N \subseteq Q$ that must be quotiented out,
and in the second step, one determines an operator system structure on the quotient space $Q/N$.

To define the subspace $N$, first recall that a seminorm on a $*$-vector space is called
a $*$-seminorm if $\|x\| = \|x^*\|$, for every $x\in Q$.
For an $x\in Q_h$, set $\| x \|_{o} = \inf\{ \lambda\geq  0: \lambda 1 \pm x\in D_1 \}$;
then $\|\cdot\|_{o}$ is a seminorm on the real vector space $Q_{h}$, called the {\em order seminorm} in \cite{pt}.
It has a minimal $\|\cdot\|_{m}$ and a maximal $\|\cdot\|_{M}$ $*$-seminorm extension to $Q$ in
the sense that for every $*$-seminorm $\|\cdot \|$ which coincides with $\|\cdot\|_{o}$ on $Q_{h}$, one has $\|\cdot\|_{m}\leq \|\cdot\| \leq  \|\cdot\|_{M}$. Moreover, $\|\cdot\|_{M} \leq 2\|\cdot\|_{m}$, and hence all  $*$-seminorms extending the order seminorm
from $Q_h$ to $Q$ are uniformly equivalent. We set $N =
\{x\in Q: \| x \|= 0 \}$, where $\|\cdot\|$ is any such seminorm. It was shown in
\cite[Proposition 4.9]{pt} that $N$ coincides with the intersection of the kernels of all states of $Q.$

The following lemma gives another useful characterization of kernels.


\begin{lemma}\label{q-lem}
Let $J$ be a closed, non-unital order ideal of an operator system $\cl S$.
Then the order seminorm on $\cl S/J$ is a norm if and only if $J$ is kernel.
\end{lemma}
\begin{proof} If the order seminorm on $\cl S/J$ is a norm then, by \cite[Proposition 4.9]{pt},
the intersection of the kernels of all states on $\cl S/J$ is $\{0\}.$
If we let $q: \cl S \to \cl S/J$ be the quotient map,
then each state on $\cl S/J$ induces a state on $\cl S$
by composition with $q.$ If we let $s_{\alpha}: \cl S \to \bb C, \alpha \in \bb{A}$, denote the
family of states obtained in this manner then, clearly, $J = \bigcap_{\alpha\in \bb{A}} \ker s_{\alpha}$.
By Proposition~\ref{eq-lem}, $J$ is a kernel.

Conversely, if $J$ is a kernel then, by Proposition~\ref{eq-lem},
there exists a family $\{s_{\alpha}: \alpha \in \bb{A} \}$
of states of $\cl S$ such that $J = \bigcap_{\alpha\in \bb{A}} \ker s_{\alpha}$.
For each $\alpha\in \bb{A}$, let $t_{\alpha} : S/J\rightarrow \bb{C}$ be the functional given by
$t_{\alpha}(q(x)) = s_{\alpha}(x)$, $x\in \cl S$.
It follows from the definition of the order structure of $\cl S/J$ that $t_{\alpha}$ is a state, for
each $\alpha\in \bb{A}$.
Moreover, $\bigcap_{\alpha\in \bb{A}} \ker t_{\alpha} = \{q(0)\}$;
in other words,
$0+J$ is the only vector of $\cl S/J$ annihilated by the order seminorm.
It follows by \cite[Proposition 4.9]{pt} that the order semi-norm is a norm.
\end{proof}




\begin{prop}\label{p_quo}  Let $\cl S$ be an operator system and let $J \subseteq \cl S$ be a kernel. If we define a family of matrix cones on $\cl S/J$ by setting
\[ C_n(\cl S/J) = \{(x_{i,j} + J) \in M_n(\cl S/J) : \forall \epsilon > 0 \ \exists k_{i,j} \in J \text{ such that } \epsilon 1 \otimes I_n +(x_{i,j} + k_{i,j} ) \in M_n(\cl S)^+ \}, \]
then $(\cl S/J, \{ C_n \}_{n=1}^{\infty})$ is a matrix ordered $*$-vector space with an Archimedean matrix unit $1+J$,
and the quotient map $q: \cl S \to \cl S/J$ is completely positive.
\end{prop}
\begin{proof}
The vector space $\cl S/J$, equipped with the family $\{D_n(\cl S/J)\}_{n=1}^{\infty}$ of matrix cones, is a
matrix ordered $*$-vector space with a matrix order unit $1 + J$.
This result now follows from \cite[Proposition 3.16]{ptt}
and Lemma \ref{q-lem}.
\end{proof}

\begin{defn}\label{d_quoos} Let $\cl S$ be an operator system and $J \subseteq \cl S$ be a kernel.
We call the operator system $(\cl S/J, \{ C_n \}_{n=1}^{\infty}, 1+J)$ defined in
Proposition \ref{p_quo} the {\bf quotient operator system.}
We call the kernel $J$ {\bf order proximinal} if
$C_1(\cl S/J) = D_1(\cl S/J)$ and {\bf completely order proximinal}
if $C_n(\cl S/J) = D_n(\cl S/J)$ for all $n\in \bb{N}.$
\end{defn}

The following proposition characterizes operator system quotients in terms of a universal property.

\begin{prop}\label{p_unive} Let $\cl S$ and $\cl T$ be operator systems and let $J$   be a kernel in $\cl S.$ If $\varphi: \cl S \rightarrow \cl T$ is a unital completely positive map with $J \subseteq  \ker(\varphi),$ then the map $\tilde{\varphi}:\cl S/J \rightarrow \cl T$ given by  $\tilde{\varphi}(x+J) = \varphi(x)$ is unital and completely positive. Moreover, if $\cl R$ is an operator system and $\psi:\cl S \to \cl R$ is a unital completely positive map, with the property that whenever $\cl T$ is an operator system and $\varphi:\cl S \to \cl T$ is a unital completely positive map with $J \subseteq \ker(\varphi),$ there exists a unique unital completely positive map $\hat{\varphi}: \cl R \to \cl T$ such that $\hat{\varphi} \circ \psi = \varphi,$ then there exists a complete order isomorphism $\gamma:\cl R \to\cl S/J$ such that $\gamma\circ \psi = q.$
\end{prop}
\begin{proof}
If $(x_{ij}+J)\in C_n(\cl S/J),$ then for every $\epsilon >0,$ we have that $\epsilon 1_{\cl T} \otimes I_n + (\varphi(x_{i,j})) \in M_n(\cl T)^+.$ Since $1_{\cl T}$ is an Archimedean matrix order unit for $\cl T,$ it follows that $(\varphi(x_{i,j})) \in M_n(\cl T)^+.$ Thus, $\tilde{\varphi}$ is a unital completely positive map.

The remaining claims follow by elementary diagram chasing.
\end{proof}

By an {\it ideal} in a C*-algebra, we shall always mean a closed, two-sided ideal.
If $\cl I$ is an ideal in a unital C*-algebra $\cl A,$ and $\cl A/\cl
I$ is the quotient C*-algebra, i.e., the C*-quotient,
then $\cl I$ is the kernel of the canonical quotient map $\pi : \cl A\rightarrow \cl A/\cl I$, and is hence a kernel
in the sense of Definition \ref{d_kern}.
Moreover, it is easily seen that $\cl A/\cl I$ is completely order isomorphic to the quotient operator system of $\cl A$ by $\cl I$
as given in Definition \ref{d_quoos}.
In particular, $D_n(\cl A/\cl I) = C_n(\cl A/\cl I),$ so ideals in C*-algebras are completely order proximinal kernels.

If $\cl S \subseteq \cl A$ is an operator system in $\cl A,$ then $J= \cl I \cap \cl S$ is the kernel of the C*-quotient map $\pi,$ restricted to $\cl S.$ Thus, $J \subseteq \cl S$ is a kernel and, by Proposition \ref{p_unive},
the restriction of $\pi$ to $\cl S$ induces an injective, unital completely positive map
$\psi: \cl S/J \to \cl A/\cl I.$
We note that $\psi$ is not always a complete order isomorphism onto its range. For an example,
let $\cl K$ be the C*-algebra of compact operators on an infinite dimensional Hilbert space $H$,
$\cl K_1 = \{\lambda I + a : a\in \cl K, \lambda \in \bb{C}\}$, and $\cl A = \cl K_1\oplus \cl K_1$
($\ell^{\infty}$-direct sum). Let $\cl I = 0\oplus \cl K_1$, so that $\cl A/\cl I \simeq \cl K_1\oplus 0$.
Let
$$\cl S = \left\{\left(\begin{matrix}
\lambda I + a & 0\\ 0 & \lambda I - a\end{matrix} \right) \in \cl A \ : \ a\in \cl K, \lambda\in \bb{C}\right\}.$$
We have that $\cl S$ is an operator subsystem of $\cl A$ and that $\cl S\cap \cl I = \{0\}$. Suppose that
$x = \left(\smallmatrix\lambda I + a & 0\\ 0 & \lambda I - a\endsmallmatrix\right) \in \cl S$. We have that
$x \in \cl S^+$ if and only if $-\lambda I \leq a \leq \lambda I$, while
$x + \cl I \in (\cl A/\cl I)^+$ if and only if $-\lambda I \leq a$. Clearly, these two conditions are not equivalent, and hence
the map $\psi : \cl S\rightarrow \cl A/\cl I$ is not a complete order isomorphism onto its range.

In Proposition \ref{simplecase} we will formulate sufficient conditions which ensure that
the map $\psi : \cl S/J\rightarrow \cl A/\cl I$ is a complete order isomorphism onto its range.
We now show instead that
for every kernel $J$ in an operator system $\cl S$,
one can find a C*-algebra $\cl A$ with $\cl S\subseteq \cl A$ and an ideal $\cl I \subseteq \cl A$
with $\cl I\cap \cl S = J$ such that the induced map $\cl S/J \rightarrow \cl A
/\cl I$ is a complete order isomorphism onto its range.

\begin{thm}\label{idealisom} Let $\cl S$ be an operator system, let $J$ be a kernel in $\cl S$ and let $\langle J \rangle$ denote the two-sided ideal in $C_u^*(\cl S)$ generated by $J.$ Then $J = \langle J \rangle \cap \cl S$ and the induced map from $\cl S/J$ to $C_u^*(\cl S)/\langle J \rangle$ is a complete order isomorphism onto its range.
\end{thm}
\begin{proof} Since $\cl S/J$ is an abstract operator system, we may embed it as an operator subsystem of some C*-algebra $\cl A.$ Let $q:\cl S \to \cl S/J$ be the quotient map. Then by the universal property of $C^*_u(\cl S)$, there exists a $*$-homomorphism $\rho:C_u^*(\cl S) \to \cl A$ extending $q.$ Let $\cl I = \ker\rho$ and note that $\cl I \cap \cl S = J.$

The canonical map $\cl S\rightarrow C_u^*(\cl S)/\cl I$ is unital and completely positive; its kernel is $J$ and hence,
by Proposition \ref{p_unive},
it induces a unital completely positive map $\nph : \cl S/J\rightarrow C_u^*(\cl S)/\cl I$. On the other hand, there is a canonical
quotient map $\psi : C_u^*(\cl S)/\cl I\rightarrow \cl A$ (recall that $\cl I = \ker \rho$). The composition $\psi\circ \nph$ is
the inclusion map and hence, by Lemma \ref{comp}, $\nph$ is a  complete order isomorphism onto its range.

Since $J \subseteq \cl I,$ we have that
$\langle J \rangle \subseteq \cl I,$ and hence $J \subseteq \langle J \rangle \cap \cl S \subseteq \cl I \cap \cl S = J$;
thus, $J = \langle J \rangle \cap \cl S.$
By Proposition \ref{p_unive},
the map $\nph_1 : \cl S/J\rightarrow C_u^*(\cl S)/\langle J \rangle$ given by $\nph_1(x + J) = x + \langle J \rangle$ is
unital and completely positive.
Since $\langle J \rangle \subseteq \cl I,$ there exists a canonical $*$-epimorphism $\psi_1 :
C_u^*(\cl S)/\langle J \rangle \rightarrow C_u^*(\cl S)/\cl I.$
However, $\psi_1\circ \nph_1 = \nph$ is, by the previous paragraph, a complete order isomorphism onto its range.
It follows from Lemma \ref{comp} that $\nph_1$ must be a complete order isomorphism onto its range.
\end{proof}

\begin{cor} Let $\cl S$ be an operator system and let $J \subseteq \cl S.$ Then $J$ is a kernel if and only if $J$ is the intersection of a closed two-sided ideal in $C_u^*(\cl S)$ with $\cl S.$
\end{cor}
\begin{proof} Suppose that $J$ is a kernel. Then, by Theorem \ref{idealisom}, we have $J= \langle J \rangle \cap \cl S.$
Conversely, if $J = \cl I \cap \cl S,$ for some ideal $\cl I\subseteq C_u^*(\cl S)$, then $J$ is the
kernel of the restriction to $\cl S$ of the quotient map from $C_u^*(\cl S)$ to $C_u^*(\cl S)/ \cl I$.
\end{proof}

If $\cl S$ is an operator system and $J\subseteq \cl S$ is a kernel,
the canonical map $\cl S/J \rightarrow C^*_u(\cl S)/\langle J\rangle$ gives rise, by the universal property of
the maximal C*-algebra, to a $*$-homomorphism $\pi : C_u^*(\cl S/J) \to C_u^*(\cl S)/\langle J \rangle$
satisfying $\pi(x+J) = x +\langle J \rangle$, for all $x \in \cl S.$

\begin{cor}\label{c_uniio} Let $\cl S$ be an operator system and let $J\subseteq \cl S$ be a kernel.
The map $\pi: C_u^*(\cl S/J) \to C_u^*(\cl S)/\langle J \rangle$ is a $*$-isomorphism.
\end{cor}
\begin{proof}
Let $q : \cl S\rightarrow \cl S/J$ be the quotient map and $\iota : \cl S/J\rightarrow C^*_u(\cl S/J)$ be the
canonical inclusion.
By the universal property of $C^*_u(\cl S/J)$, the map $\iota\circ q$ has an extension to a $*$-homomorphism
$\rho_0 : C^*_u(\cl S)\rightarrow C^*_u(\cl S/J)$ whose kernel contains $\langle J \rangle$.
Hence, $\rho_0$ induces a surjective $*$-homomorphism $\rho: C_u^*(\cl S)/\langle J \rangle \to C_u^*(\cl S/J).$
Since $\rho_0$ extends $\iota\circ q$, we have that $\rho(x+\langle J \rangle) = x+J$ for all $x \in \cl S$.
Since $\{x + \langle J \rangle : x\in \cl S\}$ (resp. $\{x + J : x\in \cl S\}$) generates
$C_u^*(\cl S)/\langle J \rangle$ (resp. $C_u^*(\cl S/J)$), we conclude that $\pi$ and $\rho$ are mutual inverses.
\end{proof}


\section{Comparison with the Operator Space Quotient}

Recall that if $X$ is an operator space and $Y$ is a closed subspace
of $X$ then the quotient $X/Y$ has a canonical operator space structure
given by assigning to $M_n(X/Y)$ the quotient norm arising from
the identification $M_n(X/Y) = M_n(X)/M_n(Y)$, that is, by setting
$$
\|(x_{ij}+Y)\|_{n} = \inf \{ \|(x_{ij}+y_{ij})\|_{n}:\;y_{ij}\in Y \}, \ \ \ (x_{ij}+Y)\in M_n(X/Y).
$$
Furthermore, if $X_1$ and $X_2$ are operator spaces,
$\phi:X_1\rightarrow X_2$ is a completely contractive map and $Y
\subseteq \ker\phi$ is a closed subspace, then the map $\tilde{\phi}:X_1/Y \rightarrow X_2$ given
by $\tilde{\phi}(x+Y)=\phi(x)$ is still completely contractive.


Suppose that $\cl S$ is an operator system and $J \subseteq \cl S$ is a kernel.
Then $\cl S/J$ can be equipped with two natural operator space
structures. Since $J$ is a closed subspace of $\cl S$, we have the {\em
quotient operator space structure} described in the previous paragraph.
On the other hand, by the results in Section \ref{s_qos},
$\cl S/J$ possesses a {\em quotient operator system structure}
which in turn induces an operator space structure as described in Section \ref{s_prel}.

We will prove shortly that the matrix norms on $\cl S/J$ obtained in
these two different fashions are in general not equal.
We will use the notation $\|\cdot\|_{osy}^{(n)}$ (respectively, $\|\cdot\|_{osp}^{(n)}$)
for the operator system quotient (respectively, operator space quotient) norm on $M_n(\cl S/J)$.
The norms $\|\cdot\|_{osy}$ and $\|\cdot\|_{osp}$ stand for $\|\cdot\|_{osy}^{(1)}$ and $\|\cdot\|_{osp}^{(1)}$, respectively.

\begin{prop}\label{p_ncha}
Let $\cl S$ be an operator system and let $J$ be a kernel in $S.$ Given $(x_{i,j}) \in M_n(\cl S),$ we have that \[\|(x_{i,j} +J) \|_{osp}^{(n)} = \sup \{ \|(\phi(x_{i,j}))\| \ : \ \phi: \cl S
\to \cl B(H), \phi(J) =\{0\}, \phi \text{ completely contractive } \},\]
and
\[ \|(x_{i,j} + J)\|_{osy}^{(n)} = \sup \{ \|(\phi(x_{i,j})\| \ : \ \phi: \cl
S \to \cl B(H), \phi(J) = \{0\}, \phi \text{ unital, completely positive
} \}, \]
where in each case $H$ runs through all Hilbert spaces.
\end{prop}
\begin{proof}
The first claim is rather well known.  We include the proof for
completeness.

Suppose that $\phi : \cl S\rightarrow \cl B(H)$ is a completely contractive map with $\phi(J) = \{0\}$.
Then $\phi$
induces a completely contractive map $\tilde{\phi}$
from the operator space quotient $\cl S/J$ into $\cl B(H)$ in the natural way.
Thus, the right hand side of the first identity does not exceed its left hand side.
The equality follows since the
(abstract) operator space quotient has a completely isometric representation into $\cl B(H)$ for some $H,$
by the Ruan representation theorem for operator spaces.

Similarly, in the second case, a unital completely positive map $\phi : \cl S\rightarrow \cl B(H)$
with $\phi(J) = \{0\}$ induces a unital completely positive (and hence completely contractive) map from the
operator system $\cl S/J$ into $\cl B(H)$.
Thus, the right hand side of the equality again does not exceed the left hand
side. The equality follows since the (abstract) operator system $\cl S/J$
has a unital complete order representation into $\cl B(H)$ for some
Hilbert space $H$, by the Choi-Effros representation theorem for operator
systems.
\end{proof}

\begin{cor}\label{normineq}
Let $\cl S$ be an operator system and let $J$ be a kernel in $\cl S.$ Then $\|
\cdot\|_{osy}^{(n)} \leq \|\cdot\|_{osp}^{(n)}$ for every $n$.
\end{cor}
\begin{proof}
The statement follows from Proposition \ref{p_ncha} and the fact that the unital, completely positive maps that vanish on
$J$ form a subfamily of the family of completely contractive maps that vanish on $J.$
\end{proof}

The following lemma will enable us to construct certain examples of operator system
and operator space quotients.

\begin{lemma}\label{ex-lemma}
Let $\cl S$ be an operator system and $y$ be a self-adjoint element of $\cl S
$ which is neither positive nor negative. Then $\sspp(y) = \{\lambda y : \lambda\in \bb{C}\}$ is the kernel
of a unital completely positive map on $\cl S.$
\end{lemma}
\begin{proof}
We may assume that $\cl S = \cl A$ is a C*-algebra.
For simplicity, set $J=\sspp(y)$.
We equip $\cl A/J$ with the cone $D_1 = D_1(\cl A/J)$; it is easy to observe that
$D_1 \cap (-D_1) = \{0\}$.
It hence suffices to show that $e+J$
is an Archimedean order unit for the ordered vector space $\cl A/J$.
Indeed, in this case, if we equip the AOU space $\cl A/J$
with its minimal operator system structure ${\rm OMIN}(\cl A/J)$ (see \cite[Definition 3.1]{ptt}); then
the quotient map $q : \cl A \rightarrow \cl A/J$, which is unital and positive, will be (unital and)
completely positive from $\cl A$ to
${\rm OMIN}(\cl A/J)$ by \cite[Theorem 3.4]{ptt}. This map still has kernel $J$ and so
$J$ is the kernel of a unital completely positive map.

It remains to show that $e+J$ is an Archimedean order unit for $\cl A/J$. Let
$x+J \in \cl A/J$ be such that
$$
\epsilon (e+J) + (x+J) \in D_1 \mbox{ for every } \epsilon\geq 0.
$$

Clearly, we may assume that $x=x^*$. The above condition
is equivalent to the condition $(\epsilon e+x)+J\in D_1$, for every $ \epsilon \geq 0.$
Thus, for every $\epsilon \geq 0$, there exists an element $\alpha_{\epsilon} \in \bb{C}$
such that $\epsilon e+x + \alpha_{\epsilon} y\geq 0$ in $\cl S$.
Clearly, $\alpha_{\epsilon}$ has to be real number.

Let $X_{\epsilon} = \{ \alpha\in \mathbb{R}:  \epsilon e+x + \alpha
y\geq 0  \}$. Then $X_\epsilon$ is a non-empty closed set for every $\epsilon > 0$ and
$X_\epsilon \subseteq X_{\delta}$ whenever $\epsilon\leq \delta$.
We will show that, for $\epsilon = 1$, $X_1$ is bounded which implies that $X_\epsilon$
is a decreasing net of non-empty compact sets and consequently has non-empty intersection.

Let $y=y_1-y_2$ be the Jordan
decomposition of $y$, that is, $y_1$ and $y_2$ are positive and $y_1y_2=0$.
Suppose that $e+x+\alpha y\geq 0$ and multiply both sides by $y_1$ from right and left; then
$y_1^2 + y_1xy_1+\alpha y_1^3 \geq 0$ or, equivalently, $\alpha
y_1^3 \geq -y_1^2 - y_1xy_1  $. This condition implies the existence of a lower bound for $\alpha$.
Similarly, by multiplying both sides by $y_2$ we obtain an upper bound for $\alpha$.
It follows that $X_1$ is a bounded set, and hence there exists
$\alpha_0\in \cap_{\epsilon > 0} X_\epsilon$. We then have
$\epsilon e + x +\alpha_0 y \geq 0$ for every $\epsilon>0$.
Thus, $x +\alpha_0 y \geq 0$ and consequently $x+J\in D_1$.
\end{proof}

The example given in the first paragraph of Section \ref{s_qos} shows that Lemma \ref{ex-lemma} does not hold if $y$ is positive (or negative).

We have seen that the operator space quotient norm is greater than the
induced operator system quotient norm. We will now show that even for finite
dimensional C*-algebras, the operator system quotient norm can, for some kernels, be much smaller than
the operator space quotient norm.
These calculations will lead to
examples for which the two norms on $\cl S/J$ are not equivalent.

\begin{exam}\label{ex_44}
{\rm Consider the C*-algebra $l_4^{\infty}$ and let $y_n=(-1,0,n,2n)$, where $n\in\bb{N}$.
Since $y_n$ is self-adjoint and neither positive nor
negative, by Lemma~\ref{ex-lemma}, $J_n = \sspp(y_n)$ is the kernel of
a unital completely positive map. Let $x=(0,1,n+1,0)$; then
$x+J_n\geq 0$ and so $\|x+J_n\|_{osy}=\inf\{ \lambda>0:\lambda
(e+J_n)-(x+J_n)\geq 0 \}$. It follows that $\|x+J_n\|_{osy}=1$.
On the other hand, it is not difficult to show that
$\|x+J_n\|_{osp}=\frac{2}{3}(n+1)$.}
\end{exam}

By using Example \ref{ex_44} and infinite direct sums, one can obtain
an example of an operator system $\cl S$ and a kernel $J\subseteq \cl S$ for
which the operator space quotient norm $\| \cdot \|_{osp}$ and the induced norm on the
operator system quotient $\|\cdot\|_{osy}$ are not equivalent.  In fact, one can achieve
more, as the following result shows.

\begin{prop}\label{p_noneqi}
There exists a C*-algebra $\cl A$ and a kernel $J$ in $\cl A,$ such that the induced norm $\|\cdot\|_{osy}$ on the operator system quotient $\cl A/J$ is not complete. If $q: \cl A \to C^*_u(\cl A)/\langle J \rangle,$ denotes the quotient map, then the range
$q(\cl A)$ of $q$ is not closed.
\end{prop}
\begin{proof} Let $\cl A_n = l_4^{\infty}$ and let $J_n \subseteq \cl A_n$ be the
kernel from Example \ref{ex_44}, $n\in \bb{N}$.
Let $\cl A= \oplus_{n \in \bb N} \cl
A_n$ and $J = \oplus_{n \in \bb N} J_n$ ($\ell^{\infty}$-direct sums).
Since $J_n$ is a kernel
in $\cl A_n$, there exists Hilbert spaces $H_n$ and completely
positive maps, $\phi_n: \cl A_n \to \cl B(H_n)$ with $J_n =
\ker\phi_n.$
Thus, if $H = \oplus_{n \in \bb N} H_n$ and $\phi: \cl A \to \cl B(H)$
is the map given by $\phi(\oplus_{n\in \bb{N}} x_n) = \oplus_{n\in \bb{N}} \phi_n(x_n)$,
then $J = \ker\phi.$

Since $J$ is a closed subspace, $\cl A/J$ with the usual operator
space quotient norm is complete. By Corollary \ref{normineq}, $\|\cdot\|_{osy}
\le \| \cdot\|_{osp}.$ Thus, if $\cl A/J$ were complete in
$\|\cdot\|_{osy},$ then the two norms would be equivalent. However, if
we let $a_m = \sum_{n \in \bb{N}} x_{m,n} \in \cl A,$ be the element given
by $x_{m,n} = 0 , m \ne n$ and $x_{m,m} = (0, \frac{1}{m+1},1,0),$
then by Example \ref{ex_44}, $\|a_m +J \|_{osy} = \frac{1}{m+1},$
while $\|a_m +J \|_{osp} = \frac{2}{3}.$

Hence, the norms are not comparable and thus $(\cl A/J,\|\cdot\|_{osy})$ is not complete.
The last statement follows from the fact that the operator system quotient
$\cl A/J$ is completely order isomorphic to $q(\cl A),$ by Theorem~\ref{idealisom}.
\end{proof}

We point out that in Proposition \ref{p_noneqi} the identity map
between the
operator system and the operator space quotients not only fails to be a
completely bounded isomorphism, but the identity is not even a bounded
isomorphism.
This observation leads to the following question:

\begin{question} Let $\cl S$ be an operator system and $J\subseteq \cl S$ be a kernel.
If the norms $\|\cdot\|_{osp}$ and $\|\cdot \|_{osy}$ are equivalent, are they completely boundedly equivalent?
\end{question}

On the other hand, the operator system and operator space quotients can often coincide, and this has ramifications for the order structure. Recall that if $X$ is a normed space and $Y \subseteq X$ is a closed subspace, then $Y$ is called {\em proximinal} provided that for every $x \in X,$ there exists
$y \in Y$ such that $\|x+y\| =\|x+Y\|,$ where the second quantity is
the norm of the element $x + Y$ of $X/Y$. When $X$ is an operator
space, the subspace $Y$ is called {\em completely proximinal} provided
that $M_n(Y)$ is proximinal in $M_n(X)$ for all $n.$

\begin{defn}
A kernel $J\subseteq \cl S$ is {\bf completely
  biproximinal} provided that:
\begin{enumerate}
\item the operator space quotient and the operator system quotient
  $\cl S/J$ are completely isometric,
\item $J$ is completely proximinal as a closed subspace of the
  operator space $\cl S,$
\item $J$ is completely order proximinal,
\item for any $(s_{i,j}+J) \in M_n(\cl S/J)^+$ there exists $(p_{i,j})
  \in M_n(\cl S)^+$ with $p_{i,j}+J = s_{i,j} + J$ and $\|(p_{i,j})\|
  = \|(s_{i,j} + J)\|_{M_n(\cl S/J)}.$\end{enumerate}
\end{defn}

\begin{prop}\label{charorderprox} Let $J$ be a completely proximinal
  kernel in $\cl S$ and assume that the operator space and operator
  system quotients $\cl S/J$ are completely isometric. Then $J$ is completely biproximinal.
\end{prop}
\begin{proof} We only consider the case $n=1$; the proofs of the remaining cases are similar. We first prove that $J$ is order proximinal. Let $s+J \in C_1(\cl S/J)$ and let $t = \|s+J\|_{osy}$; then the element $h = s - (t/2)e+J$
is hermitian and $\|h\|_{osy}= t/2.$ Since $J$ is proximinal, there exists
$x \in \cl S$ with $\|x\| = t/2$ and $x+J = h.$  Replacing $x$ by $(x+
x^*)/2,$ if necessary, we may assume that $x= x^*.$  We have that $x +
(t/2)e \in \cl S^+$ and $x+(t/2)e +J = s+J$. Thus, $s+J \in D_1(\cl S/J)$
and so $C_1(\cl S/J) = D_1(\cl S/J)$ and we have shown that $J$ is
order proximinal. Also note that $\| x+(t/2)e \| \leq \|x\|+ t/2 = t$
and, since $x+(t/2)e +J = s+J$, clearly $\| x+(t/2)e \| \geq \|s+J\| =
t$. Thus, $\| x+(t/2)e \| = t$ which shows that $J$ satisfies the
fourth property of the definition.
\end{proof}

The question about the completely bounded equivalence of the two
quotient norms can perhaps be best seen from the viewpoint of {\em relative decomposability.}

\begin{defn} Let $\cl S$ and $\cl T$ be operator systems and let $J$ be a kernel in $\cl S.$ A completely bounded map $\phi: \cl S \to \cl T$ with $\phi(J) = \{0\}$ is said to be {\bf J-decomposable} if there exist   completely positive maps, $\phi_j: \cl S \to \cl T,$ with   $\phi_j(J)=\{0\}, j=1,2,3,4$ such that $\phi= (\phi_1 - \phi_2) + i(\phi_3 - \phi_4).$ We let $D_J(\cl S, \cl T)$ denote the set of J-decomposable maps
from $\cl S$ into $\cl T$.
\end{defn}

In analogy with the work of Haagerup \cite{ha}, we can see that $\phi \in D_J(\cl S, \cl T)$
if and only if there exists completely positive
$\psi_j: \cl S \to \cl T, j=1,2$, with $\psi_1(J)= \psi_2(J) = \{0\}$, such that the map $\Phi: \cl S \to M_2(\cl T)$ given by
\[ \Phi(x) = \begin{pmatrix} \psi_1(x) & \phi(x)\\ \phi(x^*)^* &
  \psi_2(x) \end{pmatrix} \]
is completely positive. We define
\[ \|\phi\|_{J\mbox{-}dec} = \inf \max \{ \|\psi_1\|, \|\psi_2\| \} \]
where the infimum is taken over all completely positive maps $\psi_1$ and $\psi_2$ annihilating $J$
with the map $\Phi$ defined as above being completely positive.

\begin{thm}\label{th_Jdec}
Let $\cl S$ be an operator system and $J$ be a kernel in $\cl S.$  Then the following are equivalent:
\begin{enumerate}
\item for every Hilbert space $H,$ every completely bounded map from $\cl S$ into $\cl B(H)$ that
  vanishes on $J$ is $J$-decomposable,


\item there exists a constant $C$ such that for all $n$ and
all $(x_{i,j}) \in M_n(\cl S),$ we have that $\|(x_{i,j} +J) \|_{osp}
\le C \|(x_{i,j} +J) \|_{osy}.$
\end{enumerate}
Moreover, in these cases the least constant satisfying (2) is equal to
the least constant satisfying $\|\phi\|_{J\mbox{-}dec} \le C \|\phi\|_{cb}$ for all
completely bounded maps $\phi : \cl S\rightarrow \cl B(H)$ vanishing on $J$.
\end{thm}
\begin{proof}
Let $(\cl S/J)_{osp}$ be the operator space quotient,
let $q: \cl S \to (\cl S/J)_{osp}$ be the quotient map and let
$\gamma: (\cl S/J)_{osp} \to \cl B(H)$ be a completely isometric representation
whose existence is guaranteed by the Ruan representation theorem \cite{ru}.
Let $C_2$ denote the least constant satisfying (2), assuming that such a constant exists, and let
$C_1$ denote the least constant satisfying $\|\phi\|_{J\mbox{-}dec} \le C_1 \|\phi\|_{cb},$ assuming such a constant exists.

Assuming that (1) holds, we have $\gamma \circ q= (\phi_1 - \phi_2) + i(\phi_3 - \phi_4)$ where each each $\phi_j$ is completely positive with $\phi_j(J) = \{0\}.$ Hence, $\phi_j$ induces a completely positive map
$\tilde{\phi_j}: (\cl S/J)_{osy} \to \cl B(H)$, $j = 1,2,3,4$, such that
$(\tilde{\phi_1} - \tilde{\phi_2}) + i(\tilde{\phi_3} - \tilde{\phi_4}): (\cl S/J)_{osy} \to (\cl S/J)_{osp}$
is the identity map. Hence, the identity map is completely bounded and (2) follows.
Moreover, the constant in (2) is at most $\|\gamma \circ q\|_{J\mbox{-}dec}.$
Thus, $C_2 \le C_1.$

Conversely, if (2) holds, then the identity map
$\iota:(\cl S/J)_{osy} \to (\cl S/J)_{osp}$ is completely bounded.
Suppose that $\phi:\cl S \to \cl B(H)$ is completely bounded and vanishes on $J,$
and let $\tilde{\phi} : (\cl S/J)_{osp} \rightarrow \cl B(H)$ be the quotient map.

Then $\tilde{\phi}\circ\iota: (\cl S/J)_{osy} \to \cl B(H)$ is completely bounded and $\|\tilde{\phi}\circ\iota\|_{cb} \le C_2\|\phi\|_{cb}.$ By Wittstock's decomposition theorem \cite{wi}, $\tilde{\phi}\circ\iota$
can be decomposed as a sum of completely positive maps on $(\cl S/J)_{osy}$
which implies that the original map $\phi$ is $J$-decomposable.

Moreover, by the version of Wittstock's decomposition theorem given in \cite{pa}, we have that there exist completely positive maps, $\psi_j: (\cl S/J)_{osy} \to \cl B(H),$ $j=1,2$ with $\|\psi_j\| \le C_2 \|\phi\|_{cb}$ such that the map $\Phi: (\cl S/J)_{osy} \to M_2(\cl B(H))$ defined as above is completely positive. From this it follows that $\|\phi\|_{J\mbox{-}dec} \le C_2\|\phi\|_{cb},$ and so $C_1 \le C_2$ and the proof is complete.
\end{proof}


\section{Exactness}

Kirchberg \cite{k} recognized the importance of exact C*-algebras,
which now play a fundamental role in general C*-algebra theory.
Exactness was later studied from an operator space perspective and
developed further. We refer the reader to \cite[Chapter~17]{p} for the details on these topics.
Briefly, given two complete operator spaces $X$ and $Y,$ we let $X
\hat{\otimes}_{\min} Y$ denote the completion of their minimal
(called also spatial) tensor product (see \cite[Section 2.1]{p} for the definitions).  An
operator space $X$ is called {\bf exact} if for every (not necessarily unital) C*-algebra $\cl B$
and ideal $\cl I$ in $\cl B,$  the sequence
\begin{equation}\label{eq_exa}
0 \to \cl I \hat{\otimes}_{\min} X \to \cl B \hat{\otimes}_{\min} X \to (\cl B/\cl I) \hat{\otimes}_{\min} X \to 0
\end{equation}
is exact.
Here the first map is the inclusion, and the second map is $q\otimes\id$, where $q : \cl B\rightarrow \cl B/\cl I$ is
the quotient map.
Note that if we used the incomplete tensor products instead of the complete ones,
then this sequence would be trivially exact for all choices of $\cl B$ and $\cl I$.

When $X$ is a C*-algebra, then the quotient map is always surjective,
that is, the sequence is right exact, and hence $X$ is exact if and only if
$\ker (q\otimes\id) = \cl I \hat{\otimes}_{\min} X,$
that is, if the sequence is left exact.
(Recall that $\cl I \hat{\otimes}_{\min} X$ coincides with the closure of $\cl I \otimes X$ in
$\cl B \hat{\otimes}_{\min} X$, see e.g. \cite{p}.)
Note that the kernel of $q\otimes\id$ always contains $\cl I \hat{\otimes}_{\min} X.$

When $X$ is only assumed to be an operator space, both left and right exactness of this
sequence are genuine issues. However, there is always a well-defined map
$$T_X: (\cl B \hat{\otimes}_{\min} X)/(\cl I \hat{\otimes}_{\min} X) \to (\cl B/\cl I)
\hat{\otimes}_{\min} X,$$ and the exactness of the sequence (\ref{eq_exa}) is equivalent to
$T_X$ being a Banach space isomorphism for every choice of $\cl B$ and $\cl I.$
When $X$ is exact, the supremum of
$\|T_X^{-1}\|_{cb}$ over all C*-algebras $\cl B$ and ideals $\cl I\subseteq \cl B$ is
called the exactness constant of $X$ and denoted ${\rm ex}(X).$ Thus, $X$ is a 1-exact
operator space, that is, ${\rm ex}(X) = 1,$ provided that $T_X$ is a completely isometric isomorphism
for all choices of $\cl B$ and $\cl I$.

In this section, we wish to study the exactness of (\ref{eq_exa}) when
$X$ is replaced by an operator system and $\cl B$ is assumed to be a
unital C*-algebra. Observe that if $\cl I\subseteq \cl B$
is an ideal then the completed operator space tensor product $\cl I \hat{\otimes}_{\min} X$
is the same as the closure $\cl I \bar{\otimes} X$ in $\cl B \hat{\otimes}_{\min} X$. In the following a bar over the tensor product $\bar{\otimes}$
represents the closure of the algebraic tensor product in the larger space.
In the operator system setting some additional questions arise.
Firstly, if $\cl I$ is an ideal in a unital C*-algebra $\cl B$ and $\cl S$ is an operator system,
then it is not apparent that $\cl I \bar{\otimes} \cl S$ is a kernel in the operator system
$\cl B \hat{\otimes}_{\min} \cl S.$ Secondly, if $\cl I \bar{\otimes} \cl S$ is a kernel then
the operator system quotient $(\cl B \hat{\otimes}_{\min} \cl S)/(\cl I \bar{\otimes} \cl S)$
could possibly differ from the operator space quotient. However,
both of these potential difficulties can be overcome using a result from \cite{p}.

\begin{thm}\label{equ-quo}
Let $\cl A$ and $\cl B$ be unital C*-algebras, let $\cl I$ be an ideal in $\cl B$ and let $\cl S \subseteq \cl A$ be an operator system.  Then $\cl I \bar{\otimes} \cl S$ is a kernel
in $\cl B \hat{\otimes}_{\min} \cl S,$ the operator system and operator space quotients $(\cl B \hat{\otimes}_{\min} \cl S)/(\cl I \bar{\otimes} \cl S)$ are completely isometric, and the induced map $(\cl B \hat{\otimes}_{\min} \cl S)/(\cl I \bar{\otimes} \cl S) \to (\cl B  \hat{\otimes}_{\min} \cl A)/(\cl I \hat{\otimes}_{\min} \cl A)$ of the operator system quotient into the C*-quotient is a unital complete order isomorphism onto its range.
\end{thm}
\begin{proof} By \cite[Lemma~2.4.8]{p}, the induced map $\iota$ from the operator space quotient $Q(\cl S) =(\cl B \hat{\otimes}_{\min} \cl S)/(\cl I \bar{\otimes} \cl S)$ into the operator space quotient $Q(\cl A) = (\cl B \hat{\otimes}_{\min} \cl A)/(\cl I \hat{\otimes}_{\min} \cl A)$ is a complete isometry. But $Q(\cl A)$ is also a C*-algebra and the natural map $\phi$ from $\cl B \hat{\otimes}_{\min} \cl S$ to $Q(\cl A)$ is a unital completely positive map.
Thus, $\ker\phi = \cl I \bar{\otimes} \cl S$ (we omit the proof of this easy fact since
a more general argument will be given in the proof of Proposition \ref{p_514}),
and the induced map from the quotient operator system $[(\cl B \hat{\otimes}_{\min} \cl S)/(\cl I \bar{\otimes} \cl S)]_{osy}$ into $Q(\cl A)$ is a unital completely positive map whose image is $Q(\cl S).$  This shows that the operator system quotient
matrix norms are larger than the corresponding operator space quotient matrix norms and hence by
Corollary~\ref{normineq}  they are equal.

Finally, since $\iota : Q(\cl S) \rightarrow Q(\cl A)$ is a unital complete isometry, it is also a complete order embedding.
\end{proof}

Theorems \ref{th_Jdec} and \ref{equ-quo} yield the following corollary.

\begin{cor}  Let $\cl A$ and $\cl B$ be unital C*-algebras, let $\cl I$
  be an ideal in $\cl B,$ let $\cl S \subseteq \cl A$ be an
  operator system and let $H$ be a Hilbert space. Then every
  completely bounded map from $\cl B \hat{\otimes}_{\min} \cl S$
into $\cl B(H)$ that vanishes on $\cl I \hat{\otimes}_{\min} \cl S$ is
$\cl I \hat{\otimes}_{\min} \cl S$-decomposable with decomposition
constant 1.
\end{cor}

A second result from \cite{p} is related to proximinality.

\begin{prop}\label{p_comprox}
Let $\cl S$ be an operator system, $\cl B$ be a unital C*-algebra and
let $\cl I$ be an ideal in $\cl B.$ Then $\cl I \hat{\otimes}_{\min}
\cl S$ is a completely biproximinal kernel in $\cl B \hat{\otimes}_{\min} \cl S.$
\end{prop}
\begin{proof} By \cite[Lemma~2.4.7]{p}, $\cl I \hat{\otimes}_{\min} \cl S$ is completely proximinal in $\cl B \hat{\otimes}_{\min} \cl S.$ By
Theorem~\ref{equ-quo}, the operator space and operator system quotients coincide, so the result follows by Proposition~\ref{charorderprox}.
\end{proof}

\begin{defn}\label{d_lrex} We call an operator system $\cl S$ {\bf left exact,} if for every unital C*-algebra $\cl B$ and every ideal $\cl I\subseteq \cl B,$ we have that $\cl I \hat{\otimes}_{\min} \cl S$ is the kernel of the   map $q \otimes id_{\cl S}: \cl B \hat{\otimes}_{\min} \cl S \to   (\cl B/\cl I) \hat{\otimes}_{\min} \cl S,$ where $q : \cl B \to \cl B/\cl I$ is the quotient map. We call $\cl S$ {\bf right exact} provided that $q \otimes \id_{\cl S}$ is surjective for every such pair $\cl I$, $\cl B$. We call $\cl S$ {\bf exact} if it is both left and right exact. We call $\cl S$ {\bf 1-exact} if it is exact and the induced map from the operator system quotient, \[\frac{\cl B \hat{\otimes}_{\min} \cl S}{\cl I \hat{\otimes}_{\min} \cl S} \longrightarrow (\cl B/\cl I) \hat{\otimes}_{\min} \cl S,\] is a complete order isomorphism.
\end{defn}

Every C*-algebra is right exact, since in that case the quotient map
is a $*$-homomorphism with dense range and is hence surjective. Thus, a
C*-algebra is exact if and only if it is left exact. Moreover, since
$*$-isomorphisms are complete order isomorphisms, a C*-algebra is exact
if and only if it is 1-exact.

If in the definition of exactness we replace ideals in C*-algebras and
C*-quotients by kernels in operator systems and operator system quotients,
then even the operator system $\bb{C}$ of complex numbers would fail to be exact.
Indeed, the map from $\cl S = \cl S \hat{\otimes}_{\min} \bb C$
to the {\it complete} operator system $(\cl S/J)_{osy} \hat{\otimes}_{\min} \bb C$ can only be surjective
provided that $(\cl S/J)_{osy}$ is already complete in its norm.  But we saw in Proposition \ref{p_noneqi}
that this does not hold in general.

We begin with a few elementary observations.

\begin{prop} An operator system is exact (respectively, 1-exact) as an operator system
if and only if it is exact (respectively, 1-exact) as an operator space.
\end{prop}
\begin{proof} By Theorem \ref{equ-quo},
the only difference between the two definitions,
is that in the operator space definition, the C*-algebras are not required to be unital.
Thus, if $\cl S$ is exact as an operator space, then it is exact as an operator system.

Conversely, assume that $\cl S$ is exact as an operator system, let
$\cl B$ be a non-unital C*-algebra and let $\cl I\subseteq \cl
B$ be an ideal. Let $\cl B_1$ denote the unitization of $\cl B.$ Then $\cl I$ is
an ideal in $\cl B_1$ and $\cl B/\cl I$ is an ideal in $\cl
B_1/\cl I$ of co-dimension one.

Consider the following commutative diagram:

\[ \begin{array}{ccccccccc}
0 & \rightarrow & \cl I \hat{\otimes}_{\min} \cl S & \rightarrow & \cl
I \hat{\otimes}_{\min} \cl S & \rightarrow & 0 & \rightarrow & 0\\
 & & \downarrow & & \downarrow & & \downarrow & & \\
0 & \rightarrow & \cl B \hat{\otimes}_{\min} \cl S & \rightarrow & \cl
B_1 \hat{\otimes}_{\min} \cl S & \rightarrow & \bb C
\hat{\otimes}_{\min} \cl S & \rightarrow & 0 \\
  & & \downarrow & & \downarrow & & \downarrow & & \\
0 & \rightarrow & (\cl B/\cl I) \hat{\otimes}_{\min} \cl S &
\rightarrow & (\cl B_1/\cl I) \hat{\otimes}_{\min} \cl S & \rightarrow
& \bb C \hat{\otimes}_{\min} \cl S & \rightarrow & 0.
\end{array} \]

Since $\cl S$ is an exact operator system, all three rows and the last
two columns are exact. From these facts and a simple diagram chase it
follows that the first column is exact. This shows that $\cl S$ is an
exact operator space.

The statement concerning 1-exactness follows from the fact that a unital map is a
complete order isomorphism if and only if it is a complete isometry.
\end{proof}

We now wish to relate 1-exactness to a type of nuclearity introduced in \cite{kptt}.
We recall that, given two functorial operator system tensor products
$\alpha$ and $\beta$ on the category of operator systems,
an operator system $\cl S$ is called {\bf $(\alpha, \beta)$-nuclear} if $\cl S \otimes_{\alpha} \cl T = \cl S \otimes_{\beta} \cl T$ as matrix ordered spaces for every operator system $\cl T;$ that is, if the identity map on $\cl S\otimes\cl T$ is a complete order isomorphism between the
operator systems $\cl S \otimes_{\alpha} \cl T$ and $\cl S \otimes_{\beta} \cl T$.
We will say that $\cl S$ is {\bf C*-$(\alpha, \beta)$-nuclear} if $\cl S \otimes_{\alpha} \cl A = \cl S \otimes_{\beta} \cl A$ as matrix ordered spaces for every unital C*-algebra $\cl A.$  Finally, given operator systems $\cl S \subseteq \cl S_1$ and $\cl T \subseteq \cl T_1,$ we will write
\[ \cl S \otimes_{\alpha} \cl T \subseteq_{coi} \cl S_1 \otimes_{\beta} \cl T_1 \]
if the inclusion map $\iota : \cl S \otimes_{\alpha} \cl T \to \cl S_1 \otimes_{\beta} \cl T_1$ is a complete order isomorphism onto its range. In this notation, the functorial tensor product of operator systems $\linj$ introduced in \cite[Section 7]{kptt} is defined by
the relation $\cl S \otimes_{\linj} \cl T \subseteq_{coi} I(\cl S) \otimes_{\max} \cl T,$
where $I(\cl S)$ denotes the injective envelope of $\cl S.$

\begin{lemma}\label{cminel lem}
Let $\cl S \subseteq \cl B(H)$ and $\cl T$ be operator systems. Then $\cl S \otimes_{\linj} \cl T
\subseteq_{coi} \cl S \otimes_{\linj} C^*_u(\cl T)$ and $\cl S \otimes_{\linj} \cl T \subseteq_{coi} \cl B(H) \otimes_{\max} \cl T.$
\end{lemma}
\begin{proof}
By Lemma \ref{l_cintom},
$\cl B\otimes_{\comm}\cl T\subseteq_{coi} \cl B\otimes_{\max} C^*_u(\cl T)$, for every unital C*-algebra $\cl B$.
Using \cite[Theorem 6.7]{kptt}, we thus have
$$\cl S \otimes_{\linj} \cl T \subseteq_{coi} I(\cl S) \otimes_{\max} \cl T = I(\cl S) \otimes_c \cl T \subseteq_{coi}
I(\cl S) \otimes_{\max} C^*_u(\cl T). $$
Thus,
$$
\cl S \otimes_{\linj} \cl T \longrightarrow \cl S \otimes_{\linj} C^*_u(\cl T) \longrightarrow
I(\cl S) \otimes_{\max} C^*_u(\cl T)
$$
is a sequence of unital completely positive maps whose composition is a complete order isomorphism.
By Lemma \ref{comp}, the first map in the sequence is a complete order isomorphism.

To show the second inclusion, note that
$$
\cl S \otimes_{\linj} \cl T \subseteq_{coi} \cl B(H) \otimes_{\linj} \cl T = \cl B(H) \otimes_{\max} \cl T
$$
where the first inclusion follows from left injectivity of $\linj$ and second equality follows from the fact that
$\cl B(H)$ is injective.
\end{proof}

\begin{thm}\label{th_1exact}
The following properties of an operator system $\cl S$ are equivalent:
\begin{enumerate}
\item $\cl S$ is $(\min,\linj)$-nuclear;

\item $\cl S$ is C*-$(\min,\linj)$-nuclear;

\item $\cl S$ is 1-exact.
\end{enumerate}
\end{thm}
\begin{proof}
(1)$\Rightarrow$(2) is trivial.

\smallskip

(2)$\Rightarrow$(3) Let $\cl I\subseteq \cl A$ be an ideal. Recall that, by Theorem \ref{equ-quo}, $ \cl S\hat{\otimes}_{\min} \cl I =  {\cl S\bar{\otimes} \cl I}  \subseteq \cl S \hat{\otimes}_{\min} \cl A $ is a kernel and the natural map
$$
\frac{\cl S \hat{\otimes}_{\min} \cl A }{\cl S \bar{\otimes} \cl I} \rightarrow  \cl S \hat{\otimes}_{\min}
(\cl A/\cl I)
$$
is (unital and) completely positive. It hence
suffices to show that this map has a completely positive inverse.
Note that if $\cl S \subseteq \cl B(H)$ then by Theorem~\ref{equ-quo}, Lemma \ref{cminel lem},
the Remark on p. 285 of \cite{p} and the assumption,
we have the following chain of completely positive maps:
$$
\begin{array}{rcl}
\cl S \hat{\otimes}_{\min} (\cl A/\cl I) = \cl S \hat{\otimes}_{\linj} (\cl A/\cl
I) & \subseteq_{coi} & \cl B(H) \hat{\otimes}_{\max}
\cl A/\cl I\\ & = & \frac{\cl B(H) \hat{\otimes}_{\max} \cl A}{\cl B(H)\bar{\otimes} \cl I}
\longrightarrow
\frac{\cl B(H) \hat{\otimes}_{\min} \cl A }{\cl B(H)\bar{\otimes} \cl I}
\ \mbox{}_{coi}\supseteq \frac{\cl S \hat{\otimes}_{\min} \cl A }{\cl S\bar{\otimes} \cl I}.
\end{array}
$$
The composition of these maps gives the desired completely positive inverse.

\smallskip

(3)$\Rightarrow$(2)
Given a unital C*-algebra $\cl B$, there exists a free group $F$ such that the full C*-algebra $\cl C = C^*(F)$ has an ideal
$\cl I$ with $\cl B \cong \cl C / \cl I$. Assume that $\cl S\subseteq \cl B(H)$.
Using Lemma \ref{cminel lem} and the Remark on p. 285 of \cite{p}, we have
$$
\cl S \hat{\otimes}_{\linj} \cl B = \cl S \hat{\otimes}_{\linj} (\cl C/\cl I) \subseteq_{coi} \cl B(H)
\hat{\otimes}_{\max} (\cl C/\cl I) = \frac{\cl B(H) \hat{\otimes}_{\max} \cl C }{\cl B
(H) \bar{\otimes} \cl I}.
$$
Also, by Theorem \ref{equ-quo},
$$
\cl S \hat{\otimes}_{\min} \cl B = \cl S \hat{\otimes}_{\min} (\cl C/\cl I)
=
\frac{\cl S \hat{\otimes}_{\min} \cl C }{\cl S \bar{\otimes} \cl I}   \subseteq_{coi}
\frac{\cl B(H) \hat{\otimes}_{\min} \cl C }{\cl B(H) \bar{\otimes} \cl I}.
$$
By Kirchberg's Theorem \cite{ki} (see Theorem \ref{th_kir}), the C*-algebras on the right hand sides are *-isomorphic; hence
$\cl S \hat{\otimes}_{\linj} \cl B = \cl S \hat{\otimes}_{\min} \cl B$ and $\cl S$ is C*-$(\min,\linj)$-nuclear.

\smallskip

(2)$\Rightarrow$(1)
By Lemma \ref{cminel lem}, for any operator system $\cl T$ we have that
$$
\cl S \otimes_{\linj} \cl T \subseteq_{coi} \cl S \otimes_{\linj} C^*_u(\cl T)
$$
and, by the injectivity of $\min$,
$$
\cl S \otimes_{\min} \cl T \subseteq_{coi} \cl S \otimes_{\min} C^*_u(\cl T).
$$
By assumption, $\cl S \otimes_{\linj} C^*_u(\cl T) = \cl S \otimes_{\min} C^*_u(\cl T)$ and the conclusion follows.
\end{proof}

Exactness, as a property of a C*-algebra, is well known to pass to
C*-subalgebras. The following is an operator system analog of this result.

\begin{cor}
Let $\cl S$ be an operator system. If $\cl S$ is 1-exact, then every operator subsystem of $\cl S$ is 1-exact. Conversely, if every finite dimensional operator subsystem of $\cl S$ is 1-exact, then $\cl S$ is 1-exact.
\end{cor}
\begin{proof}
The statements follow from the 1-exactness criteria in Theorem \ref{th_1exact}, together with the fact
that both $\min$ and $\linj$ are left injective functorial operator system tensor products.
\end{proof}

There is another connection between exact operator spaces and exact
operator systems that we would like to point out. Recall that every
operator space $X$ gives rise \cite{pa} to a canonical operator system
$\cl S_X = \left\{\begin{pmatrix}
\lambda & x\\ y^* & \mu \end{pmatrix} : \lambda,\mu\in \bb{C}, x,y\in X\right\}.$

\begin{prop} An operator space $X$ is exact (respectively, 1-exact)
if and only if $\cl S_X$ is an exact (respectively, 1-exact) operator system.
\end{prop}
\begin{proof} The statement follows
from the identification $\cl B \hat{\otimes}_{\min} \cl S_X = \begin{pmatrix} \cl B \hat{\otimes}_{\min} \bb C & \cl B \hat{\otimes}_{\min} X \\ \cl B \hat{\otimes}_{\min} X^* & \cl B \hat{\otimes}_{\min} \bb C \end{pmatrix}$, which holds for every (not necessarily unital) C*-algebra $\cl B$, and the fact that $X$ is exact (respectively, 1-exact) if and only if $X^*$ is exact (receptively, 1-exact).
\end{proof}


Our definition and study of exactness has focused,
as in the operator space case, on the minimal tensor product. But it is also useful to consider exactness of the same sequence for some of the other functorial tensor products on operator systems that were introduced in \cite{kptt}. In the following we will show that when $\min$ is replaced by $\linj$ or $\max$
then every operator system is automatically exact.

Let $\cl I$ be an ideal in a C*-algebra $\cl A$. Then there is a net
$\{e_\alpha\}$ of positive elements in the closed unit ball of $\cl I$ such that
$$
\|e_{\alpha} b - b \| \rightarrow 0 \mbox{ for every } b\in \cl I
\mbox{ and } \|e_{\alpha} a - a e_{\alpha} \| \rightarrow 0 \mbox{ for every } a\in \cl A.
$$
Such a net always exists and is called a \textit{quasi-central
  approximate unit} for the ideal $\cl I$ in the C*-algebra $\cl A$ (see e.g. \cite{krd}).
In the following lemma we list some of the properties of $\{e_{\alpha}\}$. The
proof is essentially contained in \cite[Section 2.4]{p} but we include it for the convenience of the reader.

\begin{lemma} {\label{maxx}}
Let $\cl I$ be an ideal of a C*-algebra $\cl A$ and
$\{e_\alpha\}$ be a quasi-central approximate unit for $\cl I$. Then for any $a,b \in \cl A$ we have
\begin{enumerate}
 \item $\lim_{\alpha}\| e_\alpha a + (1-e_\alpha) b \| \leq
\max(\|a\|,\|b+\cl I\|_{\cl A/\cl I})$.

\item $\lim_{\alpha}\|a-e_{\alpha}a\| = \|a+\cl I\|_{\cl A/\cl I} = d(a,\cl I)$.
\end{enumerate}
\end{lemma}
\begin{proof}
A proof of (2) is given in \cite[Lemma 2.4.4]{p}. Using the same lemma, we have
$$
\lim_{\alpha}\| e_\alpha^{1/2} a e_\alpha^{1/2}+ (1-e_\alpha)^{1/2} b
(1-e_\alpha)^{1/2}\| \leq \max(\|a\|,\|b+ \cl I\|_{\cl A/\cl I}).
$$
Also $\lim_{\alpha} \|e_{\alpha}^{1/2}a-a e_{\alpha}^{1/2}\| = \lim_{\alpha} \|(1-e_\alpha)^{1/2}a-a(1-e_\alpha)^{1/2}\| = 0$.
(1) follows by combining this equality with the displayed inequality.
\end{proof}

Now let $X$ be an operator subspace of a C*-algebra $\cl A$, $\cl I$ be an ideal of $\cl A$
and $Y = X \cap \cl I$. Note that
$Y$ is the kernel of the quotient map $\cl A \rightarrow \cl A/\cl I$ when restricted to $X$. Consequently,
the induced map
$$
X/Y \longrightarrow \cl A / \cl I
$$
is injective and completely contractive. We will consider the following special case:

\begin{lemma}{\label{gen}}
Let $\{e_\alpha\}$ be a quasi-central approximate unit for
$\cl I\subseteq\cl A$ and let $X\subseteq \cl A$ be an operator space.
If $e_{\alpha}x\in Y$ for every $x\in X$ and for every $\alpha$,
then the induced map $X/Y \rightarrow \cl A / \cl I $ is a complete isometry. Moreover, $Y$ is a completely proximinal
subspace of $X$.
\end{lemma}

\begin{proof}
To see that the induced map is an isometry we need to show that $d(x,Y) = d(x,\cl I)$ for every
$x\in X$. Assume that $d(x,\cl I)<1$; then there is an element $y\in \cl I$
such that $\|x-y\|<1$. Note that
$$
d(x,Y) \leq \|x - \underbrace{e_\alpha x}_{\in Y}\| \leq \|\underbrace{x - e_\alpha x - y + e_\alpha y}_{(1-e_\alpha)(x-y)}\| + \|y - e_\alpha y\|  < 1+ \|y - e_\alpha y\|.
$$
Since the last term converges to one,
we obtain that if $d(x,\cl I)<1$ then $d(x,Y)\leq 1$.
This clearly shows that $d(x,\cl I)\geq d(x,Y)$. Since the converse
inequality is trivial,
we have $d(x,\cl I) = d(x,Y)$.
For the matricial level norms the same proof works if one observes that $\{diag(e_\alpha)\}$ is
a quasi-central approximate unit for $M_n(\cl I) \subset M_n(\cl A)$ and $M_n(X) \cap M_n(\cl I) = M_n(X\cap \cl I)$.

Next we show that  $Y \subseteq X$ is completely proximinal. To see that it is proximinal, fix $x$ in $X$, and
let $\epsilon > 0$. Denote by $\bar{z}$ the image of an element $z\in X$ in $X/Y$ under the quotient map.
We will show that there exists an element $x_1\in X$ with $\bar{x_1} = \bar{x}$ such that
\begin{equation}\label{eq_two}
\|x_1\| - d(x_1,Y) < \epsilon \;\; \mbox{ and } \;\; \|x_1-x\| \leq \|x\| - d(x,Y).
\end{equation}
Let
$$
x_1 = \frac{d(x,Y)}{\|x\|} e_\alpha x + (1-e_\alpha) x.
$$
Since $e_{\alpha} x\in Y$ for every $\alpha$, we have that $x_1\in X$ and
$x_1 + Y = x + Y$. Moreover, the second inequality in (\ref{eq_two})
follows from the fact that $\|e_\alpha\|\leq 1$.
On the other hand, by Lemma \ref{maxx}, $\|x_1\|$ converges to $d(x,Y)$ along $\alpha$, and
hence the first inequality in (\ref{eq_two}) for large $\alpha$.

Finally, given $x$ and $\epsilon = 1/2$ we can define $x_1$; for $x_1$ and $\epsilon = 1/4$ we can define $x_2$, and
in general, for $x_n$ and $\epsilon = 1/2^{n+1}$ we can define $x_{n+1}$.
In this manner, we obtain a Cauchy sequence in $X$ whose the limit point has desired property.
This shows that $Y$ is proximinal; the proof of complete proximinality is similar.
\end{proof}

If $\cl I \subset \cl A$ is an ideal and $\cl S \subset \cl A$ is an operator system then
$J = \cl S \cap \cl I$ is a kernel in $\cl S$. In fact $J$ is the kernel of the quotient map
$q:\cl A \rightarrow \cl A/\cl I$ when restricted to $\cl S$. Thus, the induced map
$$
\cl S / J \longrightarrow \cl A/\cl I.
$$
is unital, injective and completely positive.
By the example given before Theorem \ref{idealisom}, we know that this map, in general,
is not a complete order isomorphism.
In the next proposition we give sufficient conditions for this to happen.

\begin{prop}{\label{simplecase}}
Let $\cl I$ be an ideal of a C*-algebra $\cl A$,
$\{e_\alpha\}$ be a quasi-central approximate unit for $\cl I$ and
$\cl S\subseteq \cl A$ be an operator system. If $e_{\alpha}s$ is in $J = \cl S \cap \cl I$ for
every $s \in \cl S$ and for every $\alpha$,
then the induced map $ \cl S / J \rightarrow \cl A / \cl I $ is a (unital) complete order isomorphism.
The operator space and operator system quotients $\cl S / J$ are
completely isometric and $J$ is completely biproximinal kernel in $\cl S$.
\end{prop}
\begin{proof}
We start by showing that $(\cl S / J)_{osp}$ and $\cl S / J$ are completely isometric. By Lemma~\ref{gen},
$(\cl S / J)_{osp}\subseteq \cl A/I$ completely isometrically. Since
$$
(\cl S / J)_{osp} \longrightarrow \cl S / J \longrightarrow \cl A/\cl I
$$
is a sequence of completely contractive maps with a completely isometric
composition, the first map has to be a complete isometry. This means that the induced map $\cl S / J \rightarrow \cl A/\cl I$,
which is unital, is also a complete isometry and consequently it is a
complete order isomorphism. Thus the norms coincide on $(\cl S /
J)_{osp}$ and $\cl S / J.$ By Lemma \ref{gen}, $J$ is a completely
proximinal subspace of $\cl S$. Hence, by
Proposition~\ref{charorderprox} $J$ is a completely biproximinal
kernel in $\cl S.$
\end{proof}

The proof of the next lemma is standard and we leave it as an exercise.

\begin{lemma}\label{l_qcut}
Let $\cl I$ be an ideal of a C*-algebra $\cl A$,
$\{e_\alpha\}$ be a quasi-central approximate unit for
$\cl I$ and $\cl B$ be a C*-algebra.
If $\cl B \otimes_{\tau} \cl A$ is a
completed C*-algebra tensor product then $\{1\otimes e_{\alpha}\}_{\alpha}$ is a
quasi-central approximate unit for $\cl B \bar{\otimes} \cl I \subseteq \cl B \otimes_{\tau} \cl A$.
\end{lemma}

\begin{prop}\label{p_514}
Let $\cl A$ and $\cl B$ be unital C*-algebras, $\cl S \subseteq \cl B$ be an operator system and
$\cl I \subseteq \cl A$ be an ideal. Suppose that $\cl B \otimes_{\tau} \cl A$ is a C*-algebra tensor product and $\cl S \hat{\otimes}_{\tau_0} \cl A$ be the completed operator system tensor product arising from the inclusion
$\cl S\otimes\cl A \subseteq \cl B \otimes_{\tau} \cl A$. Then
$$
\cl S \bar{\otimes} \cl I = (\cl S \hat{\otimes}_{\tau_0} \cl A )\cap
(\cl B\bar{\otimes}\cl I),
$$
that is, $\cl S \bar{\otimes} \cl I$ is the kernel of the quotient map
$\cl B \otimes_{\tau} \cl A \rightarrow (\cl B \otimes_{\tau} \cl A) /
(\cl B\bar{\otimes} \cl I)$ when restricted to $\cl S \hat{\otimes}_{\tau_0} \cl A$.
Moreover, $\cl S \bar{\otimes} \cl I \subset   \cl S \hat{\otimes}_{\tau_0} \cl A$
is a completely biproximinal kernel and we have a unital completely order isomorphic inclusion
$$
\frac{\cl S \otimes_{\tau_0} \cl A}{\cl S \bar{\otimes} \cl I} \subseteq \frac{\cl B \otimes_{\tau} \cl A}{\cl B\bar{\otimes}\cl I}.
$$
\end{prop}

\begin{proof}
Let $\{e_\alpha\}$ be a quasi-central approximate unit for
$\cl I\subset\cl A$; by Lemma \ref{l_qcut}, $\{1\otimes e_{\alpha}\}_{\alpha}$ is a quasi-central approximate unit for
$\cl B \bar{\otimes} \cl I \subseteq \cl B \otimes_{\tau} \cl A$. We first show that
$\cl S \bar{\otimes} \cl I = (\cl S \hat{\otimes}_{\tau_0} \cl A )\cap (\cl B\bar{\otimes}\cl I)$.
Clearly, $\cl S \bar{\otimes} \cl I$ is contained in the intersection on the right.
Conversely, let $w\in (\cl S \hat{\otimes}_{\tau_0} \cl A )\cap (\cl B\bar{\otimes}I)$.
In particular, $w$ is an element of the ideal $ \cl B\bar{\otimes}\cl I $ and hence
$(1\otimes e_{\alpha})w\rightarrow w$. Since $w\in \cl S \hat{\otimes}_{\tau_0} \cl A$,
it can be approximated by a sequence $\{x_n\}\subset \cl S \otimes \cl A$. We observe that, by Lemma \ref{maxx}, $\lim_{\alpha} \|x_n-(1\otimes e_\alpha) x_n\| = \|\bar{x_n}\|$, (which tends to 0 as $n\rightarrow \infty$ since $\lim x_n = w \in \cl I$). Note that for any $\alpha$ and $n$, $(1\otimes e_{\alpha}) x_n$ belongs to $\cl S \otimes \cl I$.
Now, given $n$, choose $\alpha = \alpha_n$ large enough so that
$\|(1\otimes e_{\alpha})w - w\|\leq 1/n$ and $\|x_n-(1\otimes e_\alpha) x_n\| \leq \|\bar{x_n}\|+1/n$. Then
$$
\|w- (1\otimes e_{\alpha_n}) x_n \| \leq \| w - (1\otimes e_{\alpha_n})w   \| + \| (1\otimes e_{\alpha_n}) w - (1\otimes e_{\alpha_n})x_n \| + \|(1\otimes e_{\alpha_n})x_n - x_n\|.
$$
The term on the right hand side does not exceed
$1/n + \| w-x_n \| + \|\bar{x_n}\|+1/n$ which tends 0 as $n \rightarrow \infty$. The identity now follows.

The remaining conclusions of the proposition follow from Proposition
\ref{simplecase} applied to the C*-algebra $\cl B \otimes_{\tau} \cl
A$, its ideal $\cl B \bar{\otimes} \cl I$ and the operator subsystem
$\cl S \hat{\otimes}_{\tau_0} \cl A$.
\end{proof}

\begin{cor}
Let $\cl I$ be an ideal in a C*-algebra $\cl A$ and let $\cl S$ be an operator system. Then for any $\tau \in \{\min,\inj,\linj,\rinj,\comm=\max\}$
$\cl S \bar{\otimes} \cl I$ is
a completely biproximinal kernel in $\cl S \hat{\otimes}_{\tau} \cl A$.
\end{cor}
\begin{proof}
The statement follows from Proposition \ref{p_514} and the fact that
all of the mentioned operator system tensor products are induced by
C*-algebra tensor products.
\end{proof}

Given C*-algebras $\cl A$ and $ \cl B$ and ideal $\cl I \subseteq \cl A$ we have
the inclusion $\cl B \otimes_{\max} \cl I \subset \cl B \otimes_{\max} \cl A$.

\begin{cor}
Let $\cl S$ be an operator system and $\cl I$ be an ideal in a
C*-algebra $\cl A$. Then the following operator system and operator
space quotients coincide and all of the inclusions are completely isometric:
\begin{align*}
(1)\;\;\;\;\; &\frac{\cl S \hat{\otimes}_{\min} \cl A}{\cl S \bar{\otimes}\cl I} \subseteq
\frac{\cl B \hat{\otimes}_{\min} \cl A}{ \cl B{\hat{\otimes}_{\min}}\cl I} \mbox{ for any C*-algebra } \cl B \supseteq \cl S, \\
(2)\;\;\;\;\; & \frac{\cl S \hat{\otimes}_{\linj} \cl A}{\cl S \bar{\otimes} \cl I} \subseteq \frac{I(\cl S) \hat{\otimes}_{\max} \cl A}{I(\cl S){\hat{\otimes}_{\max}}\cl I} \mbox{ where } I(\cl S) \mbox{ is the injective envelope of } \cl S, \\
(3)\;\;\;\;\; & \frac{\cl S \hat{\otimes}_{\max} \cl A}{\cl S \bar{\otimes} \cl I} \subseteq \frac{C_u^*(\cl S) \hat{\otimes}_{\max} \cl A}{ C_u^*(\cl S) {\hat{\otimes}_{\max}}\cl I} \mbox{ where } C_u^*(\cl S) \mbox{ is the universal C*-algebra of } \cl S.
\end{align*}
\end{cor}

We recall that every C*-algebra is exact with respect to $\max$
(see \cite[Chapter 17]{p}), that is, if $\cl A$ and $\cl B$ are C*-algebras and $\cl I\subset \cl A$ is an ideal then we have a bijective unital $*$-homomorphism
$$
\frac{\cl B \hat{\otimes}_{\max} \cl A}{\cl B \bar{\otimes} \cl I} \longrightarrow \cl B  \hat{\otimes}_{\max} \cl A /\cl I.
$$
By using this result we can obtain further (1-)exactness properties for operator systems.

\begin{cor}\label{me}
Let $\cl S$ be an operator system and $\cl I$ be an ideal in a C*-algebra $\cl A$. Then the following maps
$$
\frac{\cl S \hat{\otimes}_{\linj} \cl A}{\cl S \bar{\otimes} \cl I}
\longrightarrow \cl S  \hat{\otimes}_{\linj} \cl A /\cl I  \;\;\;  \mbox{ and } \;\;\;
\frac{\cl S \hat{\otimes}_{\max} \cl A}{\cl S \bar{\otimes} \cl I} \longrightarrow \cl S  \hat{\otimes}_{\max} \cl A /\cl I
$$
are bijective unital complete order isomorphisms.
\end{cor}
\begin{proof}
The statement concerning $\linj$ follows from the inclusions
$$
\frac{\cl S \hat{\otimes}_{\linj} \cl A}{\cl S \bar{\otimes} \cl I} \subseteq
\frac{I(\cl S) \hat{\otimes}_{\max} \cl A}{I(\cl S){\hat{\otimes}_{\max}}\cl I} =
I(\cl S) \hat{\otimes}_{\max} \cl A/\cl I \supseteq \cl S \hat{\otimes}_{\linj} \cl A/\cl I.
$$
The statement concerning $\max$ can be proved similarly.
\end{proof}

\section{WEP and ($\linj, \max$)-nuclearity}

There is a natural way to extend the definition of the weak expectation property for C*-algebras to operator spaces and to operator systems.  In the case of operator spaces, this was done by Pisier \cite{p} and used in \cite{pa2}.  In this section we
study the weak expectation property in the category of operator systems and relate it to the
following property of an operator system $\cl S$:
whenever $\cl S_1$ and $\cl T$ are operator systems with $\cl S \subseteq \cl S_1$, the inclusion
$$
\cl S \otimes_{\max} \cl T \subseteq \cl S_1 \otimes_{\max} \cl T
$$
is a complete order isomorphism.
In the case of C*-algebras, these two properties are equivalent, as follows from
the work of Lance \cite{La}, which we will review below.

\begin{lemma}\label{coi}
The following properties of an operator system $\cl S$ are equivalent:
\begin{enumerate}
\item $\cl S$ is $(\linj,\max)$-nuclear;

\item for any operator systems $\cl S_1$ and $\cl T$ with $\cl S \subseteq \cl S_1$, we have
$$
\cl S \otimes_{\max} \cl T \subseteq_{coi} \cl S_1 \otimes_{\max} \cl T.
$$
\end{enumerate}
\end{lemma}
\begin{proof} (1)$\Rightarrow$(2)
Let $\cl S_1 \supseteq \cl S$ and $\cl T$ be given. Let $i: \cl S \rightarrow \cl S_1$ be the inclusion and let $\tilde{j}: \cl S_1 \rightarrow I(\cl S)$ be any unital completely positive extension of the inclusion $j: \cl S \rightarrow I(\cl S)$. We have the chain of completely
positive maps
$$
\cl S \otimes_{\linj} \cl T = \cl S \otimes_{\max} \cl T  \xrightarrow{i\otimes \id} \cl S_1 \otimes_{\max} \cl T \xrightarrow{\tilde{j}\otimes \id}  I(\cl S) \otimes_{\max} \cl T.
$$
By the definition of $\linj$,
$\cl S \otimes_{\linj} \cl T \subseteq_{coi}  I(\cl S) \otimes_{\max} \cl T,$ and hence
the composition of the two maps in the above diagram
is a complete order isomorphism. Lemma~\ref{comp} implies that $i\otimes \id$ is a complete order isomorphism.

\smallskip

(2)$\Rightarrow$(1) By the definition of $\linj$ and the hypothesis,
both
$\cl S \otimes_{\linj} \cl T$ and $\cl S \otimes_{\max} \cl T$
are operator subsystems of $I(\cl S) \otimes_{\max} \cl T$, and hence they are equal.
\end{proof}

We will shortly introduce the definition of the weak expectation property (WEP) for operator systems
and show that WEP implies $(\linj,\max)$-nuclearity.
We will need some preliminary results, as well as some notation concerning dual matrix ordered spaces
(see \cite[Section 4]{ptt}).
Suppose that $\cl S$ is a normed matrix ordered $*$-vector space. We equip its Banach space dual $\cl S^*$
with the involution $f\rightarrow f^*$, where $f^*(x) = \overline{f(x^*)}$, $x\in \cl S$.
The space $\cl S^*$ has a natural matrix order structure defined as follows:
$$
(f_{ij})\in M_n(\cl S^*)^+ \mbox{ if the map } \cl S \ni s \mapsto (f_{ij}(s)) \in M_n \mbox{ is completely positive.}
$$
We equip the second dual $\cl S^{**}$ of $\cl S$ with the matrix order structure dual to that of $\cl S^*$.
If $\cl S$ is an operator system then $\cl S^*$ does not necessarily have an order unit;
however, if $\cl S$ is a finite dimensional operator system then, by \cite[Theorem 4.4]{ce},
the dual $\cl S^*$ also has a (non-canonical) Archimedean matrix order unit and is hence an operator system.
In fact, any faithful state of $\cl S$ can serve as such a unit.
We also note that there exist infinite dimensional operator systems whose dual still has an order unit.

\begin{prop}\label{bidual}
Let $\cl S$ be an operator system.
Then the canonical inclusion $\cl S \hookrightarrow \cl S^{**}$ is a complete order isomorphism onto its range and $\cl S^{**}$ is an operator system whose Archimedean matrix order unit is the image $\hat{e}$ in $\cl S^{**}$
of the Archimedean matrix order unit $e$ of $\cl S$.
\end{prop}
\begin{proof}
It is straightforward to check that the inclusion is a complete order isomorphism.
We show that $\hat{e}$ is a Archimedean matrix order unit for $\cl S^{**}$.
Given a self-adjoint functional $F\in \cl S^{**}$ note that for any positive $f$ in $\cl S^*,$ we have
$$
(\|F\|\hat{e} + F)(f) = \|F\|f(e) + F(f) \geq \|F\|f(e) - \|F\|\|f\| \geq 0
$$
and
$$
(\|F\|\hat{e} - F)(f) = \|F\|f(e) - F(f) \geq 0
$$
since $f(e) = \|f\|$.
Thus, $\hat{e}$ is an order unit for $\cl S^{**}$.
To show that $\hat{e}$ is Archimedean, suppose that  $r\hat{e} + F \geq 0$ for every $r>0$. Then
$rf(e) + F(f) \geq 0$ for all $f\geq 0$, which
implies that $F(f)\geq 0$ for all $f\geq 0$; hence, $F\geq 0$.
Thus, $\hat{e}$ is an Archimedean order unit for $\cl S^{**}$.
Similar arguments apply to all matricial levels, and thus $\cl S^{**}$ is an operator system.
\end{proof}

The above proof shows that for self-adjoint functionals the usual norm
on $\cl S^{**}$
is greater than the order norm induced by the order unit $\hat{e}.$ In
fact, for self-adjoint functionals one can show that these norms are
equal.

The following lemma is easily checked; we omit its proof.

\begin{lemma}\label{l_eqwep}
Let $\cl S$ be an operator system, $\cl S^{**}$ be its bidual operator system and
$i:\cl S \rightarrow \cl S^{**}$ be the canonical inclusion. Then the following are equivalent:
\begin{enumerate}

\item the inclusion $i:\cl S \rightarrow \cl S^{**}$ extends to a completely positive map $\tilde{i}: I(\cl S)\rightarrow \cl S^{**}$;

\item for every operator system $\cl S\subseteq \cl T$, the map $i:\cl S \rightarrow \cl S^{**}$ extends to a completely positive map $\tilde{i}: \cl T\rightarrow \cl S^{**}$;

\item there exists an inclusion $\cl S\subseteq \cl B(H)$ such that the map $i:\cl S \rightarrow \cl S^{**}$ extends to a completely positive map $\tilde{i} : \cl B(H)\rightarrow \cl S^{**}$;

\item the inclusion $i:\cl S \rightarrow \cl S^{**}$ factors through an injective operator system by completely positive (equivalently, unital completely positive) maps, that is, there exist an injective operator system $\cl T$ and completely positive (equivalently, unital completely positive) maps $\phi_1: \cl S \rightarrow \cl T$ and $\phi_2: \cl T \rightarrow \cl S^{**}$ such that $i=\phi_2\circ \phi_1$.
\end{enumerate}
\end{lemma}

\begin{defn}\label{d_wepopsys}
We say that the operator system $\cl S$ has the weak expectation property (WEP) if it satisfies
any of the equivalent conditions of Lemma \ref{l_eqwep}.
\end{defn}

We note that condition (4) of Lemma \ref{l_eqwep}, with $\cl B(H)$ in the place of $\cl T$ and ``completely contractive''
replacing ``completely positive'' was used to define property 1-WEP for operator spaces by Pisier \cite[p. 269]{p}.

It is easy to see that a C*-algebra $\cl A$ has WEP in the classical
sense if and only if it satisfies the conditions of Definition \ref{d_wepopsys} (see \cite[Chapter 15]{p}).
Thus, Definition \ref{d_wepopsys} extends the classical notion of WEP to operator systems.

If $\cl S$ and $\cl T$ are operator systems and $f : \cl S\otimes\cl T\rightarrow \bb{C}$ is a linear functional,
let $\cl L_f : \cl S\rightarrow \cl T^*$ be the linear map given by $\cl L_f (x)(y) = f(x\otimes y)$.
It was shown by Lance \cite[Lemma 3.2]{La1} that $f$ is a positive functional on $\cl S\otimes_{\max}\cl T$ if and only if $\cl L_f$ is
a completely positive map.

\begin{lemma}\label{bidual-lem}
Let $\cl S$ and $\cl T$ be operator systems. Then
$\cl S \otimes_{\max}\cl T \subseteq_{coi} \cl S^{**} \otimes_{\max}\cl T.$
\end{lemma}
\begin{proof}
By Proposition \ref{bidual}, the canonical embedding
$i : \cl T \rightarrow \cl T^{**}$ is completely positive; thus, its adjoint
$i^*:\cl T^{***}\rightarrow \cl T^*$ is also completely positive.
Suppose that $f$ is a state on $\cl S\otimes_{\max}\cl T$. By the previous paragraph,
$\cl L_f : \cl S \rightarrow \cl T^*$ is completely positive, and hence so is its second adjoint
$\cl L_f^{**} : \cl S^{**} \rightarrow \cl T^{***}$.
It follows that the map $i^*\circ \cl L_f^{**} : \cl S^{**} \rightarrow \cl T^*$ is completely positive.
Let $g$ be the state on $\cl S^{**} \otimes_{\max}\cl T$ corresponding to $i^*\circ \cl L_f^{**}$; it is easy to see that $g$ extends $f$.
We have thus shown that every state on $\cl S\otimes_{\max}\cl T$ extends to a state on $\cl S^{**} \otimes_{\max}\cl T$.
This easily implies that
$(\cl S \otimes_{\max}\cl T)^+ = (\cl S^{**} \otimes_{\max}\cl T)^+ \cap (\cl S \otimes\cl T)$.

To establish the same identity on the matricial levels,
note that $M_n(\cl S \otimes_{\max}\cl T) = M_n\otimes_{\max} \cl S \otimes_{\max}\cl T = \cl S \otimes_{\max} M_n(\cl T)$
by the associativity and commutativity of $\max$, and hence the claim follows from the previous paragraph.
\end{proof}

\begin{cor}
Let $\cl S$ and $\cl T$ be operator systems. Then
$\cl S \otimes_{\comm}\cl T \subseteq_{coi} \cl S^{**} \otimes_{\comm}\cl T.$
\end{cor}
\begin{proof} Using Lemma \ref{bidual-lem} and Lemma \ref{l_cintom},
we obtain the following inclusions:
$$
\begin{array}{rcc}
\cl S \otimes_{\comm}\cl T & \subseteq & \cl S \otimes_{\max}C_u^*(\cl T) \\
 &  & \cap \\
\cl S^{**} \otimes_{\comm}\cl T & \subseteq & \cl S^{**} \otimes_{\max}C_u^*(\cl T).
\end{array}
$$
The claim follows.
\end{proof}

\begin{thm}\label{WEP}
Let $\cl S$ be an operator system. If $\cl S$ has WEP, then it is $(\linj,\max)$-nuclear.
\end{thm}
\begin{proof}
Let $\cl T$ be an operator system and let $j$ be the inclusion $\cl S \hookrightarrow I(\cl S)$. Since $\cl S$ has
WEP, the canonical embedding $i:\cl S \rightarrow \cl S^{**}$ extends to a unital completely positive map $\tilde{i}: I(\cl S)\rightarrow \cl S^{**}.$ Then
$$
\cl S \otimes_{\max} \cl T \xrightarrow{j\otimes \id} I(\cl S) \otimes_{\max} \cl T \xrightarrow{\tilde{i}\otimes \id}  \cl S^{**} \otimes_{\max} \cl T
$$
is a sequence of unital completely positive maps whose composition is a complete order isomorphism onto its range by Lemma~\ref{bidual-lem}.
By Lemma~\ref{comp}, $j\otimes\id$ is a complete order isomorphism onto its range.
By the definition of $\linj$, we have $\cl S \otimes_{\linj} \cl T \subseteq_{coi} I(\cl S) \otimes_{\max} \cl T$.
It follows that $\cl S \otimes_{\linj} \cl T = \cl S \otimes_{\max} \cl T$.
\end{proof}

We do not know whether the converse of Theorem \ref{WEP} holds true:

\begin{question}
Does $(\linj,\max)$-nuclearity imply WEP?
\end{question}

We will show that the converse of Theorem \ref{WEP} is
true if $\cl S$ is a C*-algebra or if it is a finite dimensional operator system.
For C*-algebras, this follows from Lance's characterization of the C*-algebras having WEP \cite[Theorem B]{La}.

\begin{thm}[Lance]\label{lance}
The following properties of a unital C*-algebra $\cl A$ are equivalent:
\begin{enumerate}
\item $\cl A$ has WEP;

\item for all unital C*-algebras $\cl A_1$ with $\cl A \subseteq \cl A_1$ and all unital C*-algebras $\cl B$ we have the C*-algebra inclusion
$$
\cl A \otimes_{\max} \cl B \subseteq \cl A_1 \otimes_{\max} \cl B;
$$

\item for all operator systems $\cl S_1$ with $\cl A \subseteq \cl S_1$ and all operator systems
$\cl T$, we have
$$
\cl A \otimes_{\max} \cl T \subseteq_{coi} \cl S_1 \otimes_{\max} \cl T.
$$
\end{enumerate}
\end{thm}
\begin{proof}
The implication (3)$\Rightarrow$(2) is trivial, (2)$\Rightarrow$(1) is a part of \cite[Theorem B]{La}, while
(1)$\Rightarrow$(3) follows from Theorem~\ref{WEP} and Lemma~\ref{coi}.
\end{proof}

Note that condition (3) is a restatement of $(\linj,\max)$-nuclearity. Thus, Lance's result yields that for unital C*-algebras, WEP and
$(\linj,\max)$-nuclearity are equivalent \cite[Proposition 7.6]{kptt}.

Recall that a von Neumann algebra $\cl M$ has WEP if and only if it is injective (see \cite[Remark 15.2]{p}).
A similar result holds for operator systems; we include the proof for the convenience of the reader.

\begin{prop}\label{p_bidin}
A bidual operator system $\cl R$ has WEP if and only if $\cl R$ is injective. Consequently, in this case, $\cl R$ is a von Neumann algebra.
\end{prop}
\begin{proof}
If $\cl R$ is injective, then the inclusion of $\cl R$ into $\cl R^{**}$ is a completely positive map of
$I(\cl R)=\cl R$ into $\cl R^{**}$ and hence $\cl R$ has WEP.
Conversely, assume that $\cl R$ has WEP. We first observe that
if $\cl R= \cl S^{**}$ and $i:\cl S^*\hookrightarrow S^{***}$ is the inclusion, then its adjoint
$p : \cl S^{****}\rightarrow S^{**}$ is a completely positive projection from $\cl R^{**}$ onto $\cl R$.
Now given $\cl T_1 \subseteq \cl T_2$ and a completely positive map $\phi:\cl T_1 \rightarrow  \cl R$,
let $\tilde{\phi}: \cl T_2 \rightarrow  I(\cl R)$ be its extension. Since $\cl R$ has WEP,
there is a completely positive map $u : I(\cl R) \rightarrow \cl R^{**}$ extending the inclusion of $\cl R$ into $\cl R^{**}$.
Then $p\circ u\circ \tilde{\phi}$ is a completely positive map
from $\cl T_2$ into $\cl R$ and extends $\phi$. Hence, $\cl R$ is injective.
Finally, every injective operator system is completely order isomorphic to a C*-algebra by a result of Choi and Effros (see e.g. \cite[Theorem 15.2]{pa}.)
\end{proof}


Note that every finite dimensional operator system is a bidual operator system since the inclusion
$\cl S \hookrightarrow \cl S^{**}$ is surjective.
Recall also that the dual $\cl S^{*}$ is also an operator system as was pointed out before Lemma~\ref{bidual}.

For finite dimensional operator systems, we have the following characterizations of WEP.

\begin{thm}\label{fd}
Let $\cl S$ be a finite dimensional operator system. Then the following are equivalent:
\begin{enumerate}
\item $\cl S$ has WEP;

\item $\cl S$ is injective;

\item $\cl S$ is $(\linj,\max)$-nuclear;

\item $\cl S$ is $(\min,\max)$-nuclear;

\item $\cl S$ is completely order isomorphic to a C*-algebra;

\item $\cl S \otimes_{\linj} \cl S^*  =  \cl S \otimes_{\max} \cl S^*$.
\end{enumerate}
\end{thm}
\begin{proof}
Since $\cl S$ is finite dimensional operator system it is a bidual operator system.
Hence, by Proposition \ref{p_bidin}, (1)$\Leftrightarrow$(2)$\Rightarrow$(5).
On the other hand, any finite dimensional C*-algebra is nuclear and so it is (min,max)-nuclear as an operator system \cite[Proposition 5.15]{kptt}. Thus, (5)$\Rightarrow$(4).
(In fact, for the same reason it is easy to show (5) implies all the other conditions.)
The implications (4)$\Rightarrow$(3)$\Rightarrow$(6) are trivial.
Hence we only need to show that (6) implies (1).

Since the map $\id : \cl S \rightarrow \cl S = \cl S^{**}$ is completely positive,
the paragraph before Lemma \ref{bidual-lem} yields
a positive linear functional $f : \cl S \otimes_{\max} \cl S^* \rightarrow \mathbb{C}$ corresponding to $\id$. Note that
$$
\cl S \otimes_{\max} \cl S^* = \cl S \otimes_{\linj} \cl S^* \subseteq I(\cl S) \otimes_{\max} \cl S^*
$$
by the left injectivity of $\linj$.
Thus, $f$ extends to a positive linear functional  $\tilde{f}$ on $I(\cl S) \otimes_{\max} \cl S^*$. Let $\varphi: I(\cl S) \rightarrow \cl S^{**}=\cl S$ be the completely positive map that corresponds to $\tilde{f}$. Clearly, $\varphi$ extends $\id$.
We showed that there is a completely positive map from $I(\cl S)$ to $\cl S$ fixing $\cl S$ elementwise; thus, $\cl S$ has WEP.
\end{proof}





\section{DCEP and $(\linj,\comm)$-nuclearity}

Since the maximal operator system tensor product and the commuting
tensor product, $\comm,$ are both extensions of the maximal C*-algebra
tensor product, one expects the families of
$(\linj, \max)$-nuclear operator systems and $(\linj,\comm)$-nuclear
operator systems to be different, but to coincide for C*-algebras. Since,
for C*-algebras, these are both characterized by WEP, one expects that
WEP should split into two different properties in the operator
system category. In this section we examine $(\linj,\comm)$-nuclearity and
prove that it is characterized by a {\it double commutant expectation
  property (DCEP).}

It is known that a C*-algebra $\cl A$ has
WEP if and only if every faithful representation $\pi: \cl A \to \cl
B(H)$ can be extended to a unital completely positive map from $I(\cl
A)$ into the double commutant $\pi(\cl A)^{\prime \prime}$ of $\pi(\cl
A)$. We will prove that the DCEP for an operator system $\cl S$ is also equivalent to the
following property:
for every operator system $\cl S_1$ with $\cl S \subseteq \cl S_1$ and every operator system,
$\cl T$ we have a complete order inclusion
$$
\cl S \otimes_{\comm} \cl T \subseteq \cl S_1 \otimes_{\comm} \cl T.
$$
We will exhibit an example that shows that such operator systems do not have to be (el,max)-nuclear or have the WEP.
Finally, we relate these ideas to Kirchberg's work on WEP and the Kirchberg Conjecture.

Let $\alpha$ and $\beta$ be operator system tensor products and $\cl S$ and $\cl T$ be operator systems.
When we write $\cl S\otimes_{\alpha = \beta}\cl T$, we will mean that the identity map on $\cl S\otimes\cl T$ is a
complete order isomorphism between
$\cl S\otimes_{\alpha} \cl T$ and $\cl S\otimes_{\beta}\cl T$.

\begin{thm}\label{elc1}
The following properties of an operator system $\cl S$ are equivalent:
\begin{enumerate}
\item $\cl S$ is $(\linj,\comm)$-nuclear, that is, for every operator system $\cl T$ we have that
$$
\cl S \otimes_{\linj} \cl T = \cl S \otimes_{\comm} \cl T;
$$

\item for any operator system $\cl S_1$ with $\cl S \subseteq \cl S_1$ and any operator system $\cl T$ we have
$$
\cl S \otimes_{\comm} \cl T \subseteq_{coi} \cl S_1 \otimes_{\comm} \cl T;
$$

\item for every operator system $\cl S_1$ with $\cl S \subseteq \cl S_1$ and any C*-algebra $\cl B$ we have
$$
\cl S \otimes_{\comm = \max} \cl B \subseteq_{coi} \cl S_1 \otimes_{\comm = \max} \cl B;
$$

\item there exists an inclusion $\cl S\subseteq \cl B(H) $ such that for every C*-algebra $\cl B$ we have
$$
\cl S \otimes_{\comm=\max} \cl B \subseteq_{coi} \cl B(H) \otimes_{\comm=\max} \cl B;
$$

\item there exists an injective operator system $\cl A$ with $\cl S \subseteq \cl A$ such that for every operator system $\cl T$ we have
$$
\cl S \otimes_{\comm=\max} \cl T \subseteq_{coi} \cl A \otimes_{\comm=\max} \cl T.
$$
\end{enumerate}
\end{thm}
\begin{proof}
The implications (2)$\Rightarrow$(3)$\Rightarrow$(4) are trivial.

\smallskip

(3)$\Rightarrow$(2)
Note that
$$
\begin{array}{rcc}
\cl S \otimes_{\comm}\cl T & \subseteq_{coi} & \cl S \otimes_{\comm=\max}C_u^*(\cl T) \\
 &  & \cap \\
\cl S_1 \otimes_{\comm}\cl T & \subseteq_{coi} & \cl S_1 \otimes_{\comm=\max}C_u^*(\cl T),
\end{array}
$$
where the inclusion on the right hand side holds by assumption.

\smallskip

(4)$\Rightarrow$(5)
Let $\cl T$ be an operator system. By Lemma \ref{l_cintom} and the assumption, we have
$$\cl S\otimes_{\comm}\cl T\subseteq_{coi} \cl S\otimes_{\max} C^*_u(\cl T)\subseteq_{coi} \cl B(H)\otimes_{\max} C^*_u(\cl T).$$
Again by Lemma \ref{l_cintom}, we have
$$\cl B(H)\otimes_{\comm} \cl T\subseteq_{coi}\cl B(H)\otimes_{\max} C^*_u(\cl T).$$
It follows that $\cl S\otimes_{\comm}\cl T\subseteq_{coi} \cl B(H)\otimes_{\comm} \cl T$ and (5)
follows from the fact that $\cl B(H)$ is an injective operator system.

\smallskip

(5)$\Rightarrow$(4) Let $\cl A\subseteq \cl B(H)$ be an injective operator system satisfying (5),
$\phi : \cl B(H)\rightarrow \cl A$ be a completely positive projection and $\cl B$ be a C*-algebra.
Suppose that $u\in (\cl B(H)\otimes_{\max}\cl B)^+\cap (\cl S\otimes\cl B)$. Then
$$u = (\phi\otimes \id)(u) \in (\cl A\otimes_{\max}\cl B)^+\cap (\cl S\otimes\cl B) = (\cl S\otimes_{\max}\cl B)^+.$$
Using the identification $M_n(\cl S\otimes_{\max}\cl B) = \cl S\otimes_{\max} M_n(\cl B)$, we obtain (4).

\smallskip

(4)$\Rightarrow$(1) Assume that $\cl S\subseteq \cl B(H)$ and let $\cl T$ be any operator system.
Using the assumption, Lemmas \ref{l_cintom} and \ref{cminel lem}, the left injectivity of $\linj$ and the proof of the implication
(5)$\Rightarrow$(4),
we obtain the following complete order inclusions and identities:
$$
\begin{array}{ccccc}
\cl S \otimes_{\linj}\cl T & \subseteq & \cl B(H) \otimes_{\linj=\max}\cl T & \subseteq & \cl B(H)\otimes_{\comm=\max}C_u^*(\cl T) \\
 &  & & & \shortparallel \\
\cl S \otimes_{\comm}\cl T & \subseteq & \cl B(H) \otimes_{\comm=\max}\cl T & \subseteq & \cl B(H) \otimes_{\comm=\max}C_u^*(\cl T).
\end{array}
$$
It follows that $\cl S \otimes_{\linj}\cl T = \cl S \otimes_{\comm}\cl T$.

\smallskip

(1)$\Rightarrow$(2)
First note that for any inclusion $\cl S\subseteq \cl B(H) $ and any operator system $\cl T$ we have
by Lemma \ref{cminel lem} that
$$
\cl S \otimes_{\comm} \cl T = \cl S \otimes_{\linj} \cl T \subseteq \cl B(H) \otimes_{\linj = \max} \cl T.
$$
Let $\cl S_1$ be an operator system containing $\cl S$ and $j: \cl S \hookrightarrow \cl S_1$
be the inclusion map
and $\tilde{i}:\cl S_1 \rightarrow \cl B(H)$ be the unital completely positive extension of the inclusion map
$i : \cl S \hookrightarrow \cl B(H)$. Then
$$
\cl S \otimes_{\comm} \cl T \xrightarrow{j\otimes \id }  \cl S_1 \otimes_{\comm} \cl T  \xrightarrow{ \tilde{i}\otimes \id }
\cl B(H) \otimes_{\comm=\max} \cl T
$$
is a sequence of completely positive maps with the property that their composition is a complete order isomorphism
onto its range. So by Lemma~\ref{comp} $j\otimes\id$ is a complete order isomorphism onto its range and (2) follows.
\end{proof}

In \cite{kptt},
the tensor products $\min,$ $\max,$ $\comm,$ $\rinj$ and $\linj$ were introduced.
It is not difficult to show that we have the following order ($\alpha\leq\beta$ means that the matricial cones
of $\beta$ are contained in the corresponding matricial cones of $\alpha$):
$$
\min\;\; \leq\;\; \linj,\rinj\;\;\leq \;\; \comm \;\; \leq \max.
$$
This means that if $\cl S$ and $\cl T$ are operator systems and $\cl S \otimes_{\linj}\cl T = \cl S \otimes_{\max}\cl T$ then necessarily $\cl S \otimes_{\linj}\cl T = \cl S \otimes_{\comm}\cl T $.
Consequently, $(\linj,\max)$-nuclearity implies $(\linj,\comm)$-nuclearity.
We point out that the converse is not true. Indeed, let $\cl S $ be the
operator subsystem of $M_3$ consisting of all 3 by 3 matrices whose
$(1,3)$ and $(3,1)$ entries are equal to 0. In \cite[Theorem 5.16]{kptt} it was shown that $\cl S$ is $(\min,\comm)$-nuclear but not $(\min,\max)$-nuclear. Consequently, $\cl S$
is $(\linj,\comm)$-nuclear but not $(\linj,\max)$-nuclear.
Since $\cl S$ is finite dimensional, Theorem~\ref{fd} implies
that it does not have WEP (note that this can also be seen directly).

Our next aim is to introduce the Double Commutant Expectation Property (DCEP) and to prove that it is equivalent to
$(\linj,\comm)$-nuclearity.
For a C*-algebra $\cl A$ and a subset $X\subseteq \cl A$, we let
$X^{\prime} = \{a\in \cl A : xa = ax, \mbox{ for all }x\in \cl S\}$ be the commutant of $X$ in $\cl A$
and $X^{\prime \prime} = (X')'$ be its double commutant in $\cl A$.
If an operator system $\cl S$ is a subsystem of two operator systems $\cl S_1$ and $\cl S_2$, we say that
a map $\phi : \cl S_1\rightarrow\cl S_2$ fixes $\cl S$ if $\phi(x) = x$ for every $x\in \cl S$.

\begin{defn}\label{dcep-defn}
We say that an operator system $\cl S$ has the {\bf double commutant
expectation property (DCEP)} provided that
 for every completely order isomorphic inclusion $\cl S \subseteq \cl B(H)$, there exists a
completely positive map $\varphi: \cl B(H) \rightarrow \cl S^{\prime \prime}$ fixing $\cl S.$
\end{defn}

A unital C*-algebra $\cl A$ has WEP if and only if it has DCEP.
Indeed, this can be shown directly using Arveson's commutant lifting
theorem for completely positive maps \cite[Theorem~1.3.1]{Ar}(see also
\cite[Theorem~12.7]{pa}). We do not give this proof here, since this
equivalence also follows from the next theorem.

\begin{thm}\label{dcep-thm}
The following properties of an operator system $\cl S$ are equivalent:
\begin{enumerate}
\item $\cl S$ has the DCEP;

\item for every inclusion $\cl S\subseteq \cl B(H)$, there exists a
completely positive map $\varphi: I(\cl S) \rightarrow \cl S^{\prime \prime}$ fixing $\cl S$;

\item there exists an injective C*-algebra $\cl B$ with $\cl S \subseteq \cl B$ such that for every
inclusion $\cl S \subseteq \cl B(H)$, there exists a completely positive map
$\varphi: \cl B \rightarrow \cl S^{\prime \prime}$ fixing $\cl S$ (where $\cl S''$ is computed in $\cl B(H)$);

\item for every injective C*-algebra $\cl A$ with $\cl S \subseteq \cl A$,
there exists a completely positive map $\varphi: \cl A \rightarrow \cl
S^{\prime \prime}$ fixing $\cl S;$

\item $\cl S$ is $(\linj,\comm)$-nuclear.
\end{enumerate}
\end{thm}



Since the $\comm$
and $\max$ tensor products coincide if one of the operator systems is
a C*-algebra, we have that a unital C*-algebra is
$(\linj,\comm)$-nuclear if and only if it is
$(\linj,\max)$-nuclear. Thus, combining the equivalences of
Theorem~\ref{lance} with Theorem~\ref{dcep-thm}, we see that a unital
C*-algebra has WEP if and only if it has DCEP. However, the seven dimensional operator system introduced before Definition~\ref{dcep-defn}
has DCEP but not WEP, so these two ways of extending the C*-algebra
definition to the operator system setting are distinct.

\begin{proof}[Proof of Theorem \ref{dcep-thm}]
We will prove Theorem \ref{dcep-thm} by establishing
the chain
$$\mbox{(2)}\Rightarrow \mbox{(5)} \Rightarrow
\mbox{(4)} \Rightarrow \mbox{(1)} \Rightarrow \mbox{(2)} \Leftrightarrow \mbox{(3)}.$$

\noindent (2)$\Rightarrow$(5).
Fix an operator system $\cl T$. First note that we can find a Hilbert space $H$ such that
$\cl B(H)$ contains $\cl S$ and $\cl T$ as operator subsystems,
$\cl S\subseteq \cl T'$ and the map
$\gamma : \cl S \otimes_{\comm} \cl T \rightarrow \cl B(H)$ given by $\gamma(x\otimes y) = xy$
is a complete order isomorphism.
Indeed, to achieve this, we may represent $C^*_u(\cl S)\otimes_{\max} C_u^*(\cl T)$ faithfully on a Hilbert space $H$, and
use the fact that
$\cl S \otimes_{\comm} \cl T \subseteq_{coi} C^*_u(\cl S)\otimes_{\max} C_u^*(\cl T)$
(see Lemma \ref{l_cintom}).
Since $\cl S\subseteq \cl B(H)$, by assumption there exists a completely positive map
$\varphi: I(\cl S) \rightarrow \cl S^{\prime \prime}$ fixing $\cl S$. Clearly,
$\cl S^{\prime \prime}$ commutes with $\cl T$,
so the map $\tilde{\varphi} : I(\cl S)\otimes_{\max} \cl T\rightarrow \cl B(H)$
given by $\tilde{\varphi}(x\otimes y) = \varphi(x)y$,
is completely positive. Hence, we have the following diagram
(where the inclusion map $j$ arises from the definition of $\linj$):
$$
\cl S\otimes_{\linj}\cl T \xrightarrow{\;j\;} I(\cl S)\otimes_{\max} \cl T
\xrightarrow{\;\tilde{\varphi}\;} \cl B(H) \hookleftarrow \cl S \otimes_{\comm}\cl T.
$$
Clearly, $\tilde{\varphi}\circ j$ is completely positive; on the other hand,
its image in $\cl B(H)$ is completely order isomorphic to $\cl S\otimes_{\comm}\cl T$.
It follows that the identity from
$\cl S\otimes_{\linj}\cl T$ to $\cl S\otimes_{\comm}\cl T$ is completely positive.
However, the identity map is trivially completely positive from
$\cl S\otimes_{\comm}\cl T$ to $\cl S\otimes_{\linj}\cl T$, hence $\cl S\otimes_{\linj} \cl T = \cl S\otimes_{\comm}\cl T$.

\smallskip

\noindent (5)$\Rightarrow$(4).
Let $\cl A$ be an injective C*-algebra with $\cl S\subseteq \cl A$.
Consider the map
$\phi: \cl S\otimes_{\linj=\comm=\max}\cl S^{\prime} \rightarrow \cl A$
given by $\phi(x\otimes y) = xy$
(note that $\cl S^{\prime}$ is a C*-algebra and hence $\comm$ and $\max$ coincide).
Clearly, $\phi$ is a unital completely positive map. Since the image space $\cl A$ is injective and
$\linj$ is a left injective tensor product,
$\phi$ extends to a unital completely positive map
$\tilde{\phi} : \cl A \otimes_{\linj=\max} \cl S^{\prime}\rightarrow\cl A$
(note that $\linj$ and $\max$ here coincide since $\cl A$ is injective).
Thus the following diagram commutes:
$$
\begin{array}{ccc}
\cl S\otimes_{\linj=\comm=\max}\cl S^{\prime} & \xrightarrow{\phi} & \cl A\\
\cap & \tilde{\phi} \nearrow     & \\
\cl A \otimes_{\linj=\max} \cl S^{\prime}. & &
\end{array}
$$
The restriction of $\tilde{\phi}$ to $\cl S^{\prime}$ is a unital $*$-homomorphism.
By Choi's theorem on multiplicative domains \cite{Ch}(see also \cite[Theorem~3.18]{pa})
$\tilde{\phi}$ is an $\cl S^{\prime}$-bimodule map.
Let $\psi: \cl A \rightarrow \cl A$ be the map given by
$\psi(a) = \tilde{\phi}(a\otimes 1)$. Clearly, $\psi$ is a unital completely positive map that fixes $\cl S.$
Also, if $y\in \cl S'$ then
$$y\psi(a) = y\tilde{\phi}(a\otimes 1) = \tilde{\phi}(a\otimes y) = \tilde{\phi}(a\otimes 1)y = \psi(a)y.$$
Hence, the image of $\psi$ is in
$\cl S^{\prime \prime}$ and so $\cl S$ satisfies (4).

\smallskip

\noindent (4)$\Rightarrow$(1) is trivial.

\smallskip

\noindent (1)$\Rightarrow$(2).
Given $\cl S\subseteq \cl B(H)$, let $i : \cl S \rightarrow \cl B(H)$ be the inclusion map
and let $\varphi: \cl B(H) \rightarrow \cl S^{\prime \prime}$
be the completely positive map fixing $\cl S.$ By the injectivity of $\cl B(H)$, the map
$i$ extends to a completely positive map $\tilde{i} : I(\cl S) \rightarrow \cl B(H)$.
Now $\varphi\circ \tilde{i}: I(\cl S) \rightarrow \cl S^{\prime \prime}$ has the desired property.

\smallskip

\noindent
(2)$\Rightarrow$(3) is trivial.

\smallskip

\noindent
(3)$\Rightarrow$(2). Given an inclusion $\cl S\subseteq \cl B(H)$, let
$\varphi: \cl B \rightarrow \cl S^{\prime \prime}$ be a completely positive map fixing $\cl S$.
The inclusion map $i:\cl S \rightarrow \cl B$ extends to a completely positive map $\tilde{i}:I(\cl S) \rightarrow \cl B$.
The map $\varphi\circ \tilde{i}: I(\cl S) \rightarrow \cl S^{\prime \prime}$ clearly satisfies the desired property.
\end{proof}

We now turn our attention to Kirchberg's characterization of the C*-algebras
having WEP (equivalently, DCEP). We will see that a similar
characterization of $(\linj,\comm)$-nuclearity (equivalently, DCEP) also holds for operator systems.  We start by recalling some definitions. If $X$ is a set let $F_X$ be the (discrete) free group generated by $X$ and
$C^*(F_X)$ be the full C*-algebra of $F_X$ (see \cite{p} for a brief introduction to group C*-algebras).
When $X$ is countably infinite then we will write $F_\infty$ in the place of $F_X$.
We also say that $F$ is free group if $F=F_X$ for some set $X$.

We recall an important theorem of Kirchberg's \cite{ki} which we will use in the sequel.

\begin{thm}[Kirchberg]
\label{th_kir}
For any free group $F$ and any Hilbert space $H$, we have
$$
C^*(F) \otimes_{\min} \cl B(H) = C^*(F) \otimes_{\max} \cl B(H).
$$
\end{thm}

\begin{lemma}\label{kir}
Let $\cl S$ be an operator system. Then the following are equivalent:
\begin{enumerate}
\item $\cl S \otimes_{\min} C^*(F) = \cl S \otimes_{\max} C^*(F) $ for every free group $F$;

\item $\cl S \otimes_{\max} C^*(F) \subseteq_{coi} \cl B(H) \otimes_{\max} C^*(F) $ for some inclusion $\cl S\subseteq_{coi} \cl B(H)$
and for every free group $F$;

\item $\cl S \otimes_{\min} C^*(F_\infty) = \cl S \otimes_{\max} C^*(F_\infty)$;

\item $\cl S \otimes_{\max} C^*(F_\infty) \subseteq_{coi} \cl B(H) \otimes_{\max} C^*(F_\infty) $ for some inclusion $\cl S\subseteq_{coi} \cl B(H).$
\end{enumerate}
\end{lemma}
\begin{proof}
(1)$\Rightarrow$(2). Using Kirchberg's Theorem \ref{th_kir}, we have
$$
\cl S \otimes_{\max} C^*(F) = \cl S \otimes_{\min} C^*(F) \subseteq \cl B(H) \otimes_{\min} C^*(F) = \cl B(H) \otimes_{\max} C^*(F).
$$

\smallskip

(2)$\Rightarrow$(1). We have
$$
\cl S \otimes_{\max} C^*(F) \subseteq_{coi} \cl B(H) \otimes_{\max} C^*(F)
\ \mbox{ and } \
\cl S \otimes_{\min} C^*(F) \subseteq_{coi} \cl B(H) \otimes_{\min} C^*(F).
$$
Since the operator systems on the right hand side coincide (Theorem \ref{th_kir}), (1) follows.

From the equivalence (1)$\Leftrightarrow$(2), it follows that (3)$\Leftrightarrow$(4). Since (1)$\Rightarrow$(3) is trivial,
it remains to show the implication (3)$\Rightarrow$(1). Note that if $X\subseteq Y$ are two
sets then $C^*(F_X)\subseteq C^*(F_Y)$ and there is a unital completely positive projection from $C^*(F_Y)$ onto $C^*(F_X)$
which is the left inverse of the inclusion $C^*(F_X)\subseteq C^*(F_Y)$. So, by using Lemma~\ref{comp} and the functoriality of max,
it is easy to show that
$$
\cl S \otimes_{\max} C^*(F_X) \subseteq_{coi} \cl S \otimes_{\max} C^*(F_Y).
$$
If $X$ is a countable set then $\cl S \otimes_{\max} C^*(F_X) \subseteq_{coi} \cl S \otimes_{\max} C^*(F_\infty)$, and a similar
inclusion holds for $\min$; thus, the result follows from our assumption.
Now let $X$ be an uncountable set.
If $\cl S \otimes_{\min} C^*(F_X) \neq \cl S \otimes_{\max} C^*(F_X) $
then we can find an element $u\in \cl S \otimes C^*(F_X)$
of the form $u = \sum_{i\in I} x_i \otimes \delta_i$, where $\delta_i$ is the element of $C^*(F_X)$ corresponding to the
generator $i\in X$, $x_i\in \cl S$ and $I\subseteq F_X$, such that
$\|u\|_{\min} \neq \|u\|_{\max}$. However, $I$ must be countable, and this leads to the inequality
$\cl S \otimes_{\min} C^*(F_I) \neq \cl S \otimes_{\max} C^*(F_I) $ which contradicts our assumption.
The proof is complete.
\end{proof}

Kirchberg has shown \cite{ki_wep} that a C*-algebra $\cl A$
has WEP if and only if $\cl A \otimes_{\min} C^*(F_\infty) = \cl A \otimes_{\max} C^*(F_\infty)$.
The following theorem is an operator system version of this result.

\begin{thm}\label{th_tenfr}
Let $\cl S$ be an operator system. Then $\cl S$ is $(\linj,\comm)$-nuclear
if and only if there exists an inclusion $\cl S\subseteq_{coi} \cl B(H)$ such that
\begin{equation}\label{eq_kir}
\cl S \otimes_{\max} C^*(F_\infty) \subseteq_{coi} \cl B(H) \otimes_{\max} C^*(F_\infty).
\end{equation}
Consequently, the statements in Theorem~\ref{elc1}, DCEP, and the statements in Lemma~\ref{kir} are all equivalent.
\end{thm}
\begin{proof}
Suppose that $\cl S$ is $(\linj,\comm)$-nuclear.
By Theorem~\ref{elc1}, there exists an inclusion $\cl S\subseteq \cl B(H)$ such that for every
C*-algebra $\cl B$ we have
$$\cl S \otimes_{\max} \cl B \subseteq_{coi} \cl B(H) \otimes_{\max} \cl B.$$
In particular, letting $\cl B = C^*(F_\infty)$ we obtain (\ref{eq_kir}).

To prove the converse, note that Lemma~\ref{kir} and the assumption  yield
the equality $\cl S \otimes_{\max} C^*(F) = \cl S \otimes_{\min} C^*(F) $ for any free group $F$.
Let $\cl B$ be a C*-algebra and let $F$ be a free group and $\cl I\subseteq C^*(F)$ be an ideal
such that $\cl B = C^*(F)/\cl I$. Set $\cl C = C^*(F)$. Using Corollary~\ref{me} and Theorem \ref{th_kir}, we have
$$
\cl S\hat{\otimes}_{\max} \cl B = \cl S\hat{\otimes}_{\max} (\cl C/\cl I) =
\frac{\cl S\hat{\otimes}_{\max} \cl C}{\cl S \bar{\otimes} \cl I}
= \frac{\cl S\hat{\otimes}_{\min} \cl C}{\cl S \bar{\otimes} \cl I}
$$
and
$$
\frac{\cl S\hat{\otimes}_{\min} \cl C}{\cl S\bar{\otimes} \cl I} \subseteq_{coi}
\frac{\cl B(H)\hat{\otimes}_{\min} \cl C}{\cl B(H)\bar{\otimes} \cl I} =
\frac{\cl B(H)\hat{\otimes}_{\max} \cl C}{\cl B(H) \bar{\otimes}\cl I} =
\cl B(H)\hat{\otimes}_{\max} (\cl C/\cl I) =
\cl B(H)\hat{\otimes}_{\max} \cl B.
$$
Thus, $\cl S$ satisfies (4) of Theorem~\ref{elc1} and so $\cl S$ is (el,c)-nuclear.
\end{proof}



\section{LLP and $(\min,\rinj)$-nuclearity}\label{s_llp}

In this section we study the $(\min,\rinj)$-nuclear operator systems. We will see that a C*-algebra is
$(\min,\rinj)$-nuclear if and only if it has the local lifting property (LLP)
(we refer the reader to \cite[Chapter 16]{p} for an introduction to
the LLP).
We recall \cite{kptt} that the operator system tensor product $\rinj$ of two operator systems
$\cl S$ and $\cl T$ is defined by the inclusion $\cl S \otimes_{\rinj} \cl T \subseteq_{coi} \cl S \otimes_{\max} I(\cl T)$.
Informally, $\rinj$ is the \lq\lq flip'' of the asymmetric tensor product $\linj$.
Note that an operator system $\cl S$ is $(\min,\rinj)$-nuclear if and only if we have the complete order inclusion
$\cl S \otimes_{\min} \cl T \subseteq_{coi} \cl S \otimes_{\max} I(\cl T)$ for every operator system $\cl T$.

\begin{thm}\label{llp}
The following properties of an operator system $\cl S$ are equivalent:
\begin{enumerate}
\item $\cl S$ is $(\min,\rinj)$-nuclear;

\item $\cl S$ is C*-$(\min\mbox{-}\rinj)$-nuclear, that is, $\cl S \otimes_{\min}\cl B = \cl S \otimes_{\rinj}\cl B$ for every C*-algebra $\cl B$;

\item $\cl S \otimes_{\min} \cl B(H) = \cl S \otimes_{\max} \cl B(H)$ for every Hilbert space $H$;

\item $\cl S \otimes_{\min}\cl T = \cl S \otimes_{\max}\cl T$ for every operator system $\cl T$ having WEP;

\item $\cl S \otimes_{\min}\cl T = \cl S \otimes_{\rinj}\cl T$ for every  finite dimensional operator system $\cl T$.
\end{enumerate}
If $\cl S$ is a C*-algebra then all of the above are also equivalent to:

\  {\rm (6)} $\cl S$ has the LLP.
\end{thm}
\begin{proof}
The implications (1)$\Rightarrow$(2)$\Rightarrow$(3) are trivial.

\smallskip

(3)$\Rightarrow$(1). Let $\cl T$ be an operator system and assume that $\cl T \subseteq \cl B(H)$. Then
$$
\cl S \otimes_{\min}\cl T \subseteq \cl S \otimes_{\min} \cl B(H)
$$
$$
\cl S \otimes_{\rinj}\cl T \subseteq \cl S \otimes_{\rinj=\max} \cl B(H)
$$
where the second inclusion follows from right injectivity of $\rinj$ (see the remark before
\cite[Theorem 7.9]{kptt}).
By assumption, the operator systems on the right hand sides are equal, and so (1) follows.

\smallskip

(4)$\Rightarrow$(3) holds because $\cl B(H)$ has WEP
(indeed; every injective operator system has WEP by Definition \ref{d_wepopsys} and Lemma \ref{l_eqwep}).

\smallskip

(1)$\Rightarrow$(4).
Recall that if $\cl T$ has WEP then it is $(\linj,\max)$-nuclear by Theorem~\ref{WEP}, and hence
$\cl T \otimes_{\linj}\cl S = \cl T \otimes_{\max}\cl S$.
It follows that
$\cl S \otimes_{\rinj}\cl T = \cl S \otimes_{\max}\cl T$ and now (4) follows by the assumption.

\smallskip

(1)$\Rightarrow$(5) is trivial.

\smallskip

(5)$\Rightarrow$(1). Suppose that (5) holds and that
$\cl S$ is not $(\min,\rinj)$-nuclear.
Then there exists an operator system $\cl T$ such that $\cl S \otimes_{\min}\cl T \neq \cl S \otimes_{\rinj}\cl T$.
Thus, there exists $n\in \bb{N}$ and an element $u\in M_n(S \otimes\cl T)$
such that $u\in M_n(S \otimes_{\min}\cl T)^+$ but $u\not\in M_n(S \otimes_{\rinj}\cl T)^+$.
Since we have the identifications $M_n(S \otimes_{\min}\cl T)^+ = (M_n(\cl S)\otimes_{\min}\cl T)^+$
and $M_n(S \otimes_{\rinj}\cl T)^+ = (M_n(\cl S)\otimes_{\rinj}\cl T)^+$, we may assume that $n = 1$.
Let $\cl T_0$ be a finite dimensional operator subsystem in $\cl T$ with $u\in \cl T_0$.

By the right injectivity of $\rinj$ and the injectivity of $\min$,
we have that $u\in (\cl S \otimes_{\min}\cl T_0)^+$ but $u\not\in(\cl S \otimes_{\rinj}\cl T_0)^+$,
a contradiction. Thus, (1) is established.

Finally, assume that $\cl S$ is completely order isomorphic to a C*-algebra.
By a result of Kirchberg \cite{ki}, a C*-algebra $\cl A$
has the LLP if and only if
$\cl A \otimes_{\min} \cl B(H) = \cl A \otimes_{\max} \cl B(H)$ for every Hilbert space $H$.
It now follows that (3) and (6) are equivalent.
\end{proof}

We now give an analogue of the LLP that is more appropriate
to our setting and parallels Ozawa's definition of 1-OLLP for operator spaces. We
shall shortly obtain an analogue of Ozawa's theorem \cite{oz}, which gives a tensor
product characterization of the operator spaces possessing 1-OLLP.

\begin{defn} Let $\cl S$ be an operator system, $\cl A$ be a unital
  C*-algebra, $\cl I$ be an ideal of $\cl A$, $q : \cl A \to \cl A/\cl I$ be the quotient map
  and $\phi: \cl S \to \cl A/\cl I$ be a unital completely positive map. We say that $\phi$ {\bf lifts
  locally,} if for every finite dimensional operator subsystem $\cl
  S_0 \subseteq \cl S,$ there exists a completely positive map
  $\psi: \cl S_0 \to \cl A,$ such that $q \circ \psi(x) = \phi(x)$ for
  every $x \in \cl S_0$.
  We say that an operator system $\cl S$ has the {\bf operator system local lifting property (OSLLP)}
  provided that for every unital C*-algebra $\cl A$ and every ideal $\cl I\subseteq \cl A$,
  every unital completely positive map $\phi: \cl S \to \cl A/\cl I$ lifts locally.
\end{defn}

\begin{remark} {\rm The local lifting $\psi$
  can, equivalently, be assumed to be a unital completely positive
  map. To see this, first suppose that $\psi$ is only a completely
  positive lifting. Then $\psi(e) = 1 +h$ where $1$ is the identity of
  $\cl A$ and $h=h^* \in \cl I.$  Choose $k=k^* \in \cl I^+,$ such that $1+h+k$ is
  positive and invertible (for example, take $k = h^-$) and let $s : \cl S \to \bb C$ be any
  state. Then
\[ \psi_1(x) = (1+h+k)^{-1/2}(\psi(x) + s(x) k)(1+h+k)^{-1/2} \]
  is a unital completely positive map that still lifts $\phi$ on $\cl
  S_0.$ This observation shows that there is a certain rigidity in the
  completely positive setting that eliminates the need for a
  definition of $\lambda$-OSLLP, where $\lambda >1.$}
\end{remark}

For the results that follow, it will be important to recall that the dual of a
finite dimensional operator system is again an operator system \cite[Theorem 4.4]{ce}.
Also recall that if $\cl S$ is a finite dimensional vector space, then there is a
canonical identification between $\cl S \otimes \cl T$ and the space of all linear maps
from $\cl S^*$ into $\cl T$ given by identifying the element $u = \sum x_i \otimes t_i \in \cl S\otimes\cl T$
with the map $L_u : \cl S^*\rightarrow \cl T$ defined by $L_u(f) = \sum f(x_i)t_i.$

\begin{lemma}\label{l_idofcp}
Let $\cl S$ and $\cl T$ be operator systems, with $\cl
  S$ finite dimensional.
  Suppose that $(u_{i,j}) \in M_n(\cl S  \otimes \cl T)$. Then
  $(u_{i,j}) \in M_n(\cl S \otimes_{\min} \cl T)^+$ if and only if the map $(L_{u_{i,j}}) : \cl S^* \to
  M_n(\cl T)$, given by $(L_{u_{i,j}}) (f) = (L_{u_{i,j}}(f))$, is completely positive.
\end{lemma}
\begin{proof} This result follows by the characterization of the
  positive elements of $M_n(\cl S \otimes_{\min} \cl T)^+$ given in
  \cite[Section~4]{kptt}.
\end{proof}

\begin{thm}\label{th_kiv}
Let $\cl S$ be an operator system. Then the following are
  equivalent:
\begin{enumerate}
\item $\cl S$ has the OSLLP;
\item $\cl S \otimes_{\min} \cl B(H) = \cl S \otimes_{\max} \cl B(H)$ for
  every Hilbert space $H.$
\end{enumerate}
\end{thm}
\begin{proof} (1)$\Rightarrow$(2). Let $F$ be a free group and $I$
be an ideal of the C*-algebra $C^*(F)$ of $F$ such that $C^*_u(\cl S)= C^*(F)/I$. Write $q:C^*(F) \to
  C^*_u(\cl S)$ for the quotient map and set $\cl A = C^*(F)$.
  Given $t \in (\cl S \otimes_{\min} \cl B(H))^+,$ choose a
  finite dimensional operator subsystem $\cl S_0 \subseteq \cl S,$ such that $t \in (\cl S_0 \otimes_{\min} \cl B(H))^+.$
  Since $\cl S$ has OSLLP, there exists a completely positive map $\psi : \cl S_0 \to C^*(F)$ such that $q \circ \psi$ coincides with the
  inclusion of $\cl S_0$ into $C^*_u(\cl S).$ Consider the following
  completely positive maps:
\begin{multline*} \cl S_0 \otimes_{\min} \cl B(H) \xrightarrow{\psi \otimes \id} C^*(F)
\otimes_{\min} \cl B(H) = C^*(F) \otimes_{\max} \cl B(H)\\ \xrightarrow{q \otimes
  \id} C^*_u(\cl S) \otimes_{\max} \cl B(H) \mbox{}_{coi}\supseteq \cl S
\otimes_{\comm} \cl B(H) = \cl S \otimes_{\max} \cl B(H),\end{multline*}
where the last identity and inclusion follow from Lemma \ref{l_cintom}.
This diagram shows that $t \in (\cl S \otimes_{\max} \cl B(H))^+.$ Thus, the identity map is
 a positive map from $\cl S \otimes_{\min}\cl  B(H)$ to $\cl S
 \otimes_{\max} \cl B(H).$  A similar argument shows that it is completely positive and the implication follows.

\smallskip

(2)$\Rightarrow$(1). Let $\phi: \cl S \to \cl A/\cl I$ be a
completely positive map into a quotient C*-algebra, and fix a finite dimensional
operator subsystem $\cl S_0 \subseteq \cl S.$ Let $H$ be a Hilbert space such that
$\cl S_0^* \subseteq \cl B(H).$  The completely positive
inclusion map of $\cl S_0$ into $\cl S$ corresponds to a positive
element $u \in \cl S_0^* \otimes_{\min} \cl S \subseteq \cl B(H)
\otimes_{\min} \cl S = \cl B(H) \otimes_{\max} \cl S.$  By the exactness of
$\max$ (see Proposition \ref{me}),
we have the following chain of maps

\begin{multline*} \cl B(H) \hat{\otimes}_{\max} \cl S \xrightarrow{\id \otimes
    \phi} \cl B(H) \hat{\otimes}_{\max} (\cl A/\cl I) = \frac{\cl B(H) \hat{\otimes}_{\max}
    \cl A}{\cl B(H)\bar{\otimes} \cl I} \xrightarrow{\pi} \frac{\cl B(H)
    \hat{\otimes}_{\min} \cl A}{\cl B(H)\bar{\otimes} \cl I}, \end{multline*}
where $\pi$ is the map arising from the natural surjection
$\cl B(H) \hat{\otimes}_{\max}\cl A \rightarrow\cl B(H) \hat{\otimes}_{\min}\cl A$.
By Theorem~\ref{equ-quo},
we have that
\[ \frac{\cl S_0^* \hat{\otimes}_{\min} \cl A}{\cl S_0^*\bar{\otimes} \cl I}
\subseteq_{coi} \frac{\cl B(H) \hat{\otimes}_{\min} \cl A}{\cl B(H) \bar{\otimes} \cl I}.\]
Thus, the image $v$ of $u$ under this chain of maps is a positive
element in $\frac{\cl B(H) \hat{\otimes}_{\min} \cl A}{\cl B(H) \bar{\otimes} \cl I}$.
By Proposition~\ref{p_comprox},
$\cl S_0^* \bar{\otimes} \cl I$ is completely order
proximinal in $\cl S_0^* \hat{\otimes}_{\min} \cl A.$ Thus, $v$
can be lifted to a positive element $w$ of
$\cl S_0^* \hat{\otimes}_{\min} \cl A.$ The element $w$ corresponds to a completely
positive map $\psi$ from $\cl S_0$ into $\cl A.$  Diagram chasing
shows that $\psi$ is the desired completely positive lifting of $\phi.$

\end{proof}

Combining Theorems~\ref{llp} and \ref{th_kiv}, we see that we can add
the OSLLP to the list of properties characterizing
$(\min,\rinj)$-nuclearity. We present a property stronger than OSLLP
in the following proposition.

\begin{prop} Let $\cl S$ be an operator system. If $\cl S_0^*$ is
  1-exact for every finite
  dimensional operator subsystem $\cl S_0 \subseteq \cl S,$ then $\cl
  S$ has the OSLLP.
\end{prop}
\begin{proof} Given a C*-algebra $\cl A$, an ideal $\cl I\subseteq \cl A$, a unital completely positive map
 $\phi : \cl S \to \cl A/\cl I$ and a finite dimensional operator subsystem $\cl S_0 \subseteq \cl S,$
 we need to show that the restriction $\phi_0$ of $\phi$ to $\cl S_0$ lifts to a completely positive map.
By Lemma \ref{l_idofcp}, $\phi_0$ can be identified
  with a positive element of $\cl S_0^* \otimes_{\min} (\cl A/\cl I),$ while
  any lifting of $\phi_0$ can be identified with a positive element of $\cl S_0^*
  \otimes_{\min} \cl A.$

Since $\cl S_0^*$ is 1-exact, $\cl S_0^* \otimes_{\min} (\cl A/\cl I)$ is completely order
isomorphic to the operator system quotient
\[\frac{\cl S_0^* \otimes_{\min} \cl A}{\cl S_0^* \otimes \cl I}. \]
Hence, by the complete proximinality of $\cl S_0^*\otimes \cl I$ (Proposition~\ref{p_comprox}),
the positive element of the quotient corresponding to $\phi_0$
can be lifted to a positive element of $\cl S_0^* \otimes_{\min} \cl A$ corresponding to a completely positive lifting
of $\phi_0$.
\end{proof}

\section{An Operator System Approach to the Kirchberg Conjecture}

A famous conjecture of Kirchberg \cite{ki} states that $C^*(F_{\infty})$
has the WEP. We shall refer to this as the Kirchberg Conjecture (KC).
One equivalent formulation of the KC is that
every C*-algebra with the LLP has the WEP. The work of several authors has
lead to various operator space equivalences of the KC. See the book of
Pisier \cite[Chapter~16]{p} for an excellent summary of all these results. In this section
we obtain various operator system equivalences of the KC.

Recall that DCEP (equivalently, $(\linj,\comm)$-nuclearity) of an operator system is a property which coincides with WEP when restricted to C*-algebras. Similarly, the properties in Theorem~\ref{llp} coincide with LLP when the operator system is
in fact a C*-algebra.
It is hence natural to seek an analogue of the KC for operator systems
that involves the DCEP and OSLLP.

 The next theorem gives one operator system version of the KC.

\begin{thm}\label{oskc}
The following are equivalent:
\begin{enumerate}
\item $C^*(F_\infty)$ has WEP;

\item every $(\min,\rinj)$-nuclear operator system is $(\linj,\comm)$-nuclear;

\item for any $(\min,\rinj)$-nuclear operator system $\cl S$, one has
$$
\cl S \otimes_{\min} \cl S = \cl S \otimes_{\comm} \cl S;
$$
\item every operator system with the OSLLP has the DCEP.
\end{enumerate}
\end{thm}
\begin{proof}
(2)$\Rightarrow$(1). The C*-algebra $C^*(F_\infty)$ has LLP
(see Theorems \ref{th_kir} and \ref{th_kiv}), and hence by
Theorem \ref{llp} it is $(\min,\rinj)$-nuclear. By assumption, $C^*(F_\infty)$ is
$(\linj,\comm)$-nuclear and hence it has WEP (see Theorem \ref{WEP} and \cite[Proposition 7.6]{kptt}).

\smallskip

(1)$\Rightarrow$(2). Let $\cl S$ be a $(\min,\rinj)$-nuclear operator system.
By Theorem \ref{llp}, $\cl S \otimes_{\min} \cl T =  \cl S \otimes_{\max} \cl T$ for any operator system $\cl T$ having WEP. In particular, $\cl S \otimes_{\min} C^*(F_{\infty}) =  \cl S \otimes_{\max} C^*(F_{\infty})$.
By Lemma \ref{kir} and Theorem \ref{th_tenfr}, $\cl S$ is $(\linj,\comm)$-nuclear.

\smallskip

(2)$\Rightarrow$(3). Suppose that $\cl S$ is $(\min,\rinj)$-nuclear. Then it is also $(\linj,\comm)$-nuclear, and
hence $\cl S\otimes_{\linj}\cl S = \cl S\otimes_{\comm}\cl S$. After flipping, we obtain
$\cl S\otimes_{\rinj}\cl S = \cl S\otimes_{\comm}\cl S$ and hence we have that
$\cl S\otimes_{\min}\cl S = \cl S\otimes_{\comm}\cl S$. Hence (3) holds.

\smallskip

(3)$\Rightarrow$(1). The implication follows
from the fact that KC is equivalent to the equality
$$
C^*(F_\infty) \otimes_{\min} C^*(F_\infty) = C^*(F_\infty) \otimes_{\max} C^*(F_\infty)
$$
and the fact that $\comm$ and $\max$ coincide for C*-algebras (Lemma
\ref{l_cintom} (1)).

Finally, the equivalence of (2) and (4) follows from the fact that an
operator system is $(\min, \rinj)$-nuclear if and only if it has the
OSLLP by Theorem~\ref{th_kiv} and Theorem~\ref{llp} combined with the
fact that an operator system is $(\linj, \comm)$-nuclear if and only
if it has the DCEP by Theorem~\ref{dcep-thm}.
\end{proof}

\begin{defn}
We will say that an operator subsystem $\cl S \subseteq \cl A$ of a
unital C*-algebra {\bf contains enough
unitaries} if the unitaries in $\cl S$ generate $\cl A$ as a
C*-algebra.
\end{defn}

We will see that such an operator system contains considerable information about the C*-algebra.
The following lemma is an analog of \cite[Proposition 13.6]{p} in the operator system setting.

\begin{lemma}\label{uniqueextlem}
Let $\cl S \subseteq \cl A$ contain enough unitaries and let $\phi:\cl S \rightarrow \cl B$ be
a unital completely positive map. Suppose that $\{u_\alpha\}$ is a collection of unitaries in $\cl S$ which generate $\cl A.$ If $\phi(u_\alpha)$ is a unitary in $\cl B$ for every $\alpha$ then $\phi$ extends to a unital $*$-homomorphism from $\cl A$ to $\cl B$.
\end{lemma}

\begin{proof}
This is another application of Choi's theory of multiplicative
domains \cite{Ch}. Since $1= \phi(u_{\alpha}^*u_{\alpha}) =
\phi(u_{\alpha})^*\phi(u_{\alpha}),$ each $u_{\alpha}$ belongs to the
multiplicative domain of $\phi,$ and since these elements generate
$\cl A,$ $\phi$ is a $*$-homomorphism.
\end{proof}

\begin{lemma}
Suppose $\cl A$ and $\cl B$ are unital C*-algebras and the operator system $\cl S\subseteq \cl A$ contains enough unitaries.
If $\cl A \otimes_{\tau} \cl B$ is a C*-algebra tensor product then $\cl S \otimes \cl B$ contains enough unitaries in $\cl A \otimes{_\tau} \cl B$.
\end{lemma}

\begin{proof}
Let $\{u_{\alpha}\}_{\alpha}$ be the collection of unitaries in $\cl
S$ and $\{v_{i}\}_i$ be the set of unitaries in $\cl B$. Then it is
easy to see that $\{u_{\alpha}\otimes v_i\}\subseteq \cl S \otimes \cl
B$ generates $\cl A \otimes_{\tau} \cl B$ as a C*-algebra.
\end{proof}

\begin{prop}{\label{unitarysystem}}
Let the operator system $\cl S \subseteq \cl A$ contain enough
unitaries in $\cl A$ and let $\cl B$ be a unital C*-algebra. Then
\begin{enumerate}
\item $\cl S \otimes_{\min} \cl B \subseteq \cl A \otimes_{\max} \cl B \;\;\;\Longrightarrow \;\;\; \cl A \otimes_{\min} \cl B = \cl A \otimes_{\max} \cl B$.

\item $\cl S \otimes_{\min} \cl B \subseteq \cl A \otimes_{\linj} \cl B \;\;\;\Longrightarrow \;\;\; \cl A \otimes_{\min} \cl B = \cl A \otimes_{\linj} \cl B$.

\item $\cl S \otimes_{\linj} \cl B \subseteq \cl A \otimes_{\max} \cl B \;\;\;\Longrightarrow \;\;\; \cl A \otimes_{\linj} \cl B = \cl A \otimes_{\max} \cl B$.
\end{enumerate}
\end{prop}
\begin{proof}
The proofs of all the assertions are based on the same idea so will show only (1). Let $\{u_{\alpha}\}_{\alpha}$ be the collection of unitaries in $\cl S$ and $\{v_{i}\}_i$ be the set of unitaries in $\cl B$. So $\{u_{\alpha}\otimes v_i\}\subseteq \cl S \otimes \cl B $ generates $ \cl A \otimes_{\min} \cl B $ and the inclusion $\cl S \otimes_{\min} \cl B \hookrightarrow \cl A \otimes_{\max} \cl B$ is
a unital completely positive map that maps $u_{\alpha}\otimes v_i$ to a unitary.
By Lemma \ref{uniqueextlem} it has an extension that is a unital $*$-homomorphism from $\cl A \otimes_{\min} \cl B$ into
$\cl A \otimes_{\max} \cl B$. This map has to be identity so the result follows.
\end{proof}

Before stating the following corollary we remind the reader that for an operator system $1$-exactness and $(\min,\linj)$-nuclearity coincide and for a C*-algebra being exact in the classical sense is same as being 1-exact as an operator system. Similarly, $(\linj,\comm)$-nuclearity
is equivalent to DCEP and for a C*-algebra WEP and DCEP coincide.

\begin{cor}{\label{enoghunitar}}
Assume that the operator system $\cl S \subseteq \cl A$ contains enough unitaries.
\begin{enumerate}
 \item If $\cl S$ is $(\min,\comm)$-nuclear then $\cl A$ is nuclear.

 \item If $\cl S$ is $(\min,\linj)$-nuclear then $\cl A$ is exact.

 \item If $\cl S$ is $(\linj,\comm)$-nuclear then $\cl A$ has WEP.
\end{enumerate}
\end{cor}

\begin{proof}
(1) If $\cl S$ is $(\min,\comm)$-nuclear then the condition in Proposition
\ref{unitarysystem}(1) will be automatically satisfied for every
unital C*-algebra $\cl B$. Thus, $\cl A$ is nuclear and (1) follows.

\smallskip

(2) Suppose that $\cl S$ is $(\min,\linj)$-nuclear, and let $\cl B$ be a C*-algebra. Then
$\cl S\otimes_{\min} \cl B = \cl S\otimes_{\linj} \cl B$ and since $\linj$ is left injective,
we have that $\cl S\otimes_{\linj} \cl B \subseteq \cl A\otimes_{\linj} \cl B$. Thus,
$\cl S\otimes_{\min} \cl B  \subseteq \cl A\otimes_{\linj} \cl B$. It now follows from
Proposition \ref{unitarysystem} (2) that $\cl A\otimes_{\min} \cl B  = \cl A\otimes_{\linj} \cl B$.
Thus, $\cl A$ is C*-$(\min,\linj)$-nuclear, which by Theorem \ref{th_1exact} implies that $\cl A$ is exact.

\smallskip

(3) Suppose that $\cl S$ is $(\linj,\comm)$-nuclear, and let $\cl B$
be a C*-algebra. Let $\tau$ be the operator system structure on $\cl
S \otimes \cl B$ arising from its inclusion into $\cl A
\otimes_{\max} \cl B$. We will first show that $\linj\leq \tau \leq
\comm$. Let $i : \cl S\rightarrow \cl A$ be the inclusion map. By
the functoriality of $\comm$, the map $i\otimes \id : \cl S
\otimes_{\comm} \cl B \rightarrow \cl A\otimes_{\max} \cl B$ is
completely positive; thus, $\tau \leq \comm$. On the other hand,
$\id\otimes \id: \cl A\otimes_{\max} \cl B \rightarrow \cl
A\otimes_{\linj} \cl B$ is completely positive, by the maximality of
$\max$. Since $\cl A\otimes_{\linj} \cl B$ (resp. $\cl
A\otimes_{\max} \cl B$) contains completely order isomorphically the
operator system $\cl S\otimes_{\linj} \cl B$ (resp. $\cl
S\otimes_{\tau} \cl B$), we have that $\linj\leq \tau$. Now, since
$\cl S$ is $(\linj,\comm)$-nuclear, the condition in Proposition
\ref{unitarysystem} (3) is satisfied and consequently $\cl A$ is
C*-$(\linj,\max)$-nuclear, or equivalently, by \cite[Proposition
7.6]{kptt}, $\cl A$ has WEP.
\end{proof}

Let $\cl S_\infty$ be the smallest closed operator subsystem in the full C*-algebra $C^*(F_{\infty})$ containing all the generators, that is,
$$
\cl S_{\infty} = \overline{span}\{g_{i}^*, e, g_{i}: i=1,2,3,\dots\},
$$
where $g_i,\;\;i=1,2,3,\dots$, are the free unitary generators of $ C^*(F_{\infty})$. We define $\cl
S_n$ similarly, that is, $\cl S_n$ is the $2n+1$ dimensional operator system in the full C*-algebra $ C^*(F_{n}) $ containing the  unitary generators. Both $\cl S_\infty$ and $\cl S_n$ contain enough unitaries in the corresponding C*-algebras. We wish to study these operator systems to explore the  properties of $ C^*(F_{\infty})$ and $ C^*(F_{n})$.

\begin{prop}{\label{char. lemm.}}
Let $\alpha$ be $\infty$ or a natural number $n$. Let $\cl T$ be an
operator system and  let $f:\{g_{i}\}_{i=1}^{\alpha} \rightarrow \cl T$ be a function such that $\|f(g_i)\|\leq 1$ for every $i$.
Then there is a unique, unital completely positive map
$\phi:\cl S_\alpha \rightarrow \cl T$ with $\phi(g_i) = f(g_i)$ for
every $i$.
\end{prop}

In other words, every contractive function defined from generators to an operator system extends uniquely to a
unital completely positive map.

\begin{proof}
The proof is based on the fact that every contraction has a unitary dilation.
Let $\cl T \subseteq \cl B(H)$. So we have a set of contractions
$\{f(g_i)\}_i \subseteq \cl B(H)$. By Sz.-Nagy's dilation theorem,
there exists a Hilbert space $K_i = H\oplus H_i$ and a unitary $u_i$ in $\cl B(K_i)$ such that
$f(g_i) = P_H u_i |_H$. Now let $K = H\oplus (\oplus_{i} H_i)$ and $U_i$ is the operator in $\cl B(K)$ defined as $u_i$ on $H\oplus H_i$
and identity on $(H\oplus H_i)^{\perp}$. Clearly, $\{U_i\} \subseteq \cl B(K)$ and
$P_H U_i |_H = f(g_i)$ for every $i$. Let $\tilde{f}:\{g_i\} \rightarrow \cl B(K) $ be given by $\tilde{f}(g_i) = U_i$.
By the universal property of the free group $F_{\alpha}$, $\tilde{f}$
extends to a group homomorphism $\rho$ from $F_{\alpha}$ into the unitary group of $\cl B(K)$.
Similarly, the universal property of the full C*-algebra $C^*(F_\alpha)$ ensures that there exists a unital $*$-homomorphism $\pi: C^*(F_\alpha)\rightarrow \cl B(K) $ with $\pi(g_i) = \rho(g_i)$.
Note that $\phi = P_H \pi(\cdot)|_H$ is a unital completely positive map from
$C^*(F_\alpha)$ to $\cl B(H)$ such that $\phi(g_i) = f(g_i)$ for each $i$.
Now the restriction of $\phi$ to $\cl S_\alpha$
clearly has the same property and its image stays in $\cl T$. 
\end{proof}

It is important to observe that the property given in Proposition \ref{char. lemm.} characterizes the operator system
$\cl S_\alpha$, $\alpha = 1,2,\dots,\infty$:

\begin{lemma}
Let $\cl S$ be a norm complete operator system and let
$\{s_i\}_{i=1}^\alpha$ be a set in the unit ball of $\cl S$ such that
every function $f:\{s_i\}_{i=1}^\alpha \rightarrow \cl T$ with
$\|f(s_i)\|\leq 1$ extends uniquely to a unital completely positive
map from $\cl S$ into $\cl T.$
Then the unique unital completely positive extension $\psi$ of
$f: \{s_i\}_{i=1}^\alpha \rightarrow \cl S_\alpha$ given by $f(s_i) = g_i$ is a bijective complete order isomorphism.
\end{lemma}
\begin{proof}
Let $\phi:\cl S_{\alpha} \rightarrow \cl S$ be the unique
completely positive extension of the map $h: \{g_i\}_{i=1}^\alpha \rightarrow \cl S$ given by $h(g_i) = s_i$. It is easy to see that
$$
\cl S_\alpha \xrightarrow{\phi} \cl S \xrightarrow{\psi} \cl S_{\alpha}
$$
is a sequence of unital completely positive maps whose composition is the
identity on $\cl S_{\alpha}$. So the first map has to be complete
order isomorphism. We will show that it is surjective. Let $\cl S_0 =
\overline{span}\{e,s_i,s_i^*\}_i$. By the assumption every unital completely positive
map from $\cl S_0$ extends uniquely to $\cl S$. This means that $\cl S_0 = \cl S$ and thus $\phi$ is surjective.
\end{proof}

\begin{prop}{\label{snhaslp}}
$\cl S_\alpha$ has the OSLLP, for every $\alpha = 1,2,\dots,\infty$.
\end{prop}
\begin{proof}
Let $\cl I$ be an ideal in a unital C*-algebra $ \cl A$ and let
$\phi: \cl S_\alpha \rightarrow \cl A/\cl I$ be a unital completely positive map.
Since $\cl I\subseteq \cl A$ is proximinal, for each $i$
there exists $a_i$ in $\cl A$ such that $a_i+\cl I = \phi(g_i)$ and
$\|a_i\| = \|\phi(g_i)\|_{\cl A/\cl I} \le 1.$
Let $\tilde{\phi}: \cl S_\alpha \rightarrow \cl A$ be the unique unital completely positive
extension of the function $f$ given by $f(g_i) = a_i$ given by
Proposition~\ref{char. lemm.}. Clearly, $\tilde{\phi}$ is a completely positive lifting of $\phi$.
\end{proof}

We now examine the nuclearity properties of $\cl S_{\infty}$ and of
the finite dimensional
operator systems $\cl S_{n}.$ Let $i_n:\cl
S_n\rightarrow \cl S_\infty$ be the unique unital completely positive
extension of the function $f:\{g_1,...,g_n\}\rightarrow \cl S_\infty$
given by sending
$g_i$ to itself for $i=1,...,n$ and let $p_n: \cl S_\infty\rightarrow \cl S_n$ be the
unique unital completely positive extension of
$h:\{g_1,g_2,...\}\rightarrow \cl S_n$ given by $h(g_i) = g_i$ for
$i=1,...,n$ and $h(g_i) =0$ for $i >n.$ Note that
$$
\cl S_n\xrightarrow{i_n} \cl S_\infty \xrightarrow{p_n}\cl S_n
$$
is a sequence of unital completely positive maps such that $p_n\circ
i_n$ is the identity on $\cl S_n$. This means that $i_n$ is an
complete order isomorphism. Consequently $\cl S_\infty$ contains $\cl
S_n$ completely order isometrically in a natural way and there is projection from $ \cl S_\infty$ onto $\cl S_n$.

\begin{lemma} The sequence of maps
$$
\cl S_\infty \xrightarrow{p_n}\cl S_n \xrightarrow{i_n} \cl S_\infty
$$
converges to the identity on $\cl S_{\infty}$ in the point-norm topology.
\end{lemma}

\begin{proof} We use the fact that the span of the generators is
  isometrically isomorphic to $\ell_1.$ See for example
  \cite[Exercise~14.3]{pa} or \cite[Theorem~9.6.1]{p}.
Fix
$$
s = \sum_{i=1}^{\infty}\beta_ig_i^* + \beta e + \sum_{i=1}^{\infty}\theta_ig_i.
$$
Note that
$$
i_np_n(s) = \sum_{i=1}^{n}\beta_ig_i^* + \beta e + \sum_{i=1}^{n}\theta_ig_i.
$$
So $s-i_np_n(s) = \sum_{i=n+1}^{\infty}\beta_ig_i^* + \sum_{i=n+1}^{\infty}\theta_ig_i $. Since both sum stays in the $\ell_1$ portion of the space we have
$$
\|s-i_np_n(s)\| \leq \sum_{i=n+1}^\infty|\beta_i| + \sum_{i=n+1}^\infty|\theta_i| \xrightarrow{n} 0.
$$
\end{proof}

\begin{lemma}{\label{inlusionsn}}
Let $\cl T$ be an operator system and let $\tau$ be a functorial operator system tensor product. Then $$
\cl S_n \otimes_\tau \cl T \subseteq_{coi} \cl S_\infty \otimes_\tau \cl T.
$$
\end{lemma}

\begin{proof}
$ \cl S_n \otimes_\tau \cl T \xrightarrow{i_n \otimes id}  \cl
S_\infty \otimes_\tau \cl T  \xrightarrow{p_n \otimes id} \cl S_n
\otimes_\tau \cl T$ is a sequence of unital completely positive maps
and the composition is identity on the first space. So the first map
is a complete order inclusion.
\end{proof}

\begin{thm}{\label{finiteinf}}
Let $\tau_1$ and $\tau_2$ be functorial operator system tensor
products with $\tau_1 \leq  \tau_2$. Then $\cl S_\infty$ is
$(\tau_1,\tau_2)$-nuclear if and only if $\cl S_n$ is $(\tau_1,
\tau_2)$-nuclear for every $n$.
\end{thm}
\begin{proof}
Using the maps $i_n$ and $p_n,$ we see that
$$
\cl S_n \otimes_{\tau_1} \cl T \subseteq_{coi} \cl S_\infty \otimes_{\tau_1} \cl T.
$$
and
$$
\cl S_n \otimes_{\tau_2} \cl T \subseteq_{coi} \cl S_\infty \otimes_{\tau_2} \cl T.
$$
Thus,
$(\tau_1, \tau_2)$-nuclearity of $\cl S_\infty$ implies the same
property for $\cl S_n$ for every $n$.

Now we will show the converse. Consider the sequence of unital
completely positive maps
$$
\cl S_\infty \otimes_{\tau_1} \cl T \xrightarrow{p_n\otimes id} \cl S_n\otimes_{\tau_1} \cl T = \cl S_n \otimes_{\tau_2} \cl T \xrightarrow{i_n\otimes id} \cl S_\infty\otimes_{\tau_2} \cl T.
$$
The composition of these maps converges to the identity map from  $\cl
S_\infty \otimes_{\tau_1} \cl T$ to $\cl S_\infty \otimes_{\tau_2} \cl
T$ in the point norm toplogy. (First look at the elementary tensors
and use cross norm property.) This clearly implies that the
identity is a unital completely positive map, and the proof is complete.
\end{proof}

\begin{thm}
The following are equivalent:
\begin{enumerate}
 \item $C^*(F_\infty)$ has WEP,

 \item $\cl S_\infty$ is $(\linj,\comm)$-nuclear,

 \item $\cl S_n$ is $(\linj,\comm)$-nuclear for every natural number $n$.
\end{enumerate}
\end{thm}

\begin{proof}
Clearly (2) and (3) are equivalent by Theorem~\ref{finiteinf}. Since
unitaries in $\cl S_\infty$ generate $C^*(F_\infty)$, by
Corollary~\ref{enoghunitar}, (2) implies (1).

Finally assume (1). By Proposition~\ref{snhaslp} $\cl S_{\infty}$ has the OSLLP, so by
Theorem~\ref{oskc}, $\cl S_{\infty}$ has DCEP, i.e., is $(\linj,\comm)$-nuclear.
\end{proof}




Recall that KC is equivalent to the identity $C^*(F_\infty ) \otimes_{\min} C^*(F_\infty ) = C^*(F_\infty ) \otimes_{\max} C^*(F_\infty )$.
\begin{thm}
The following are equivalent:
\begin{enumerate}
 \item $C^*(F_\infty)$ has WEP.

 \item $\cl S_\infty \otimes_{\min} \cl S_\infty = \cl S_\infty \otimes_{\comm} \cl S_\infty $

 \item $\cl S_n \otimes_{\min} \cl S_n = \cl S_n \otimes_{\comm} \cl S_n $ for every $n$.
\end{enumerate}
\end{thm}

\begin{proof} To see that (3) implies (2) note that the sequence of
  unital completely positive maps
$$
\cl S_\infty \otimes_{\min} \cl S_\infty \xrightarrow{p_n\otimes p_n}  \cl S_n \otimes_{\min} \cl S_n = \cl S_n \otimes_{c} \cl S_n \xrightarrow{i_n\otimes i_n} \cl S_\infty \otimes_{c} \cl S_\infty
$$
approximate the identity in the point norm topology. (First consider
the elementary tensors and use the cross property of norms.) Thus,
$\cl S_\infty \otimes_{\min} \cl S_\infty \xrightarrow{id} \cl S_\infty
\otimes_{c} \cl S_\infty $ is completely positive. This proves that (3) implies (2).

To show that (2) implies (1), note first that $\cl S_\infty \otimes_{\min} \cl S_\infty$
contains enough unitaries in $C^*(F_\infty ) \otimes_{\min}
C^*(F_\infty )$, namely the set $\{g_i\otimes g_j\}_{i,j \in
  \mathbb{Z}}$ where $g_0$ represents $e$ and $g_{-n} = g_{n}^*$. Let
$\tau$ be the operator system tensor product on $\cl S_\infty
\otimes \cl S_\infty$ arising from the inclusion into $C^*(F_\infty)
\otimes_{\max} C^*(F_\infty )$. Clearly, $\min \leq \tau \leq
\comm$. Our assumption now implies the inclusion
$$
\cl S_\infty \otimes_{\min} \cl S_\infty \hookrightarrow C^*(F_\infty ) \otimes_{\max} C^*(F_\infty).
$$
Note that this inclusion maps $\{g_i\otimes g_j\}_{i,j \in \mathbb{Z}}$ to unitaries. So it extends to a unital $*$-homomorphism from $C^*(F_\infty ) \otimes_{\min} C^*(F_\infty ) $ to $ C^*(F_\infty ) \otimes_{\max} C^*(F_\infty )$. This shows that (2) implies (1).

(1) implies (3) by Theorem~\ref{oskc}.
\end{proof}

\end{document}